\documentclass[12pt]{amsart}

\usepackage{amsthm,amssymb,enumitem,mathtools}

\usepackage{tikz}

\usepackage{scalerel}
\usepackage{hyperref}

\usetikzlibrary{arrows,shapes,automata,backgrounds,decorations,petri,positioning}

\tikzstyle{edge} = [fill,opacity=.5,fill opacity=.5,line cap=round, line join=round, line width=50pt]

\pgfdeclarelayer{background}
\pgfsetlayers{background,main}

\usepackage[margin=1in,letterpaper,portrait]{geometry}

\usepackage{amsthm}

\usepackage[latin1]{inputenc}
\usepackage{subfigure}
\usepackage{color}
\usepackage{amsmath}
\usepackage{amsthm}
\usepackage{amstext}
\usepackage{amssymb}
\usepackage{amsfonts}
\usepackage{graphicx}
\usepackage{young}
\usepackage{multicol}
\usepackage{mathrsfs}
\usepackage[all]{xy}
\usepackage{tikz}
\usepackage{mathtools}
\usepackage{tabularx}
\usepackage{array}
\usepackage{commath}
\usepackage{ytableau}

\theoremstyle{plain}
\theoremstyle{definition}
\newtheorem{theorem}{Theorem}[section]

\newtheorem{lemma}[theorem]{Lemma}

\newtheorem{definition}[theorem]{Definition}

\newtheorem{question}[theorem]{Question}
\newtheorem{requirement}[theorem]{Requirement}
\newtheorem{example}[theorem]{Example}
\newtheorem{proposition}[theorem]{Proposition}
\newtheorem{corollary}[theorem]{Corollary}

\DeclareMathAlphabet{\mathpzc}{OT1}{pzc}{m}{it}

\newcommand{\scale}[1]{\scalebox{1}{$\scriptscriptstyle ({#1})$}}
\newcommand{\longelt}[1]{w_0^{\scale{#1}}}

\newcommand{\perim}{\textsf{perim}}
\newcommand{\leftperim}{\textsf{L-perim}}
\newcommand{\rightperim}{\textsf{R-perim}}
\newcommand{\topperim}{\textsf{T-perim}}
\newcommand{\bottomperim}{\textsf{B-perim}}

\newcommand{\omitt}[1]{}

\usepackage{amssymb,amsmath,tabularx,graphicx}
\usepackage{tikz}
\usetikzlibrary[calc,intersections,through,backgrounds,arrows,decorations.pathmorphing]

\begin{document}

\title{Tiling-based models of perimeter and area}

\author{Bridget Eileen Tenner}
\address{Department of Mathematical Sciences, DePaul University, Chicago, IL, USA}
\email{bridget@math.depaul.edu}
\thanks{Research partially supported by a DePaul University Faculty Summer Research Grant and by Simons Foundation Collaboration Grant for Mathematicians 277603.}



\begin{abstract}
We consider polygonal tilings of certain regions and use these to give intuitive definitions of tiling-based perimeter and area. We apply these definitions to rhombic tilings of Elnitsky polygons, computing sharp bounds and average values for perimeter tiles in convex centrally symmetric $2n$-gons. These bounds and values have implications for the combinatorics of reduced decompositions of permutations. We also classify the permutations whose polygons gave minimal perimeter, defined in two different ways. We conclude by looking at some of these questions in the context of domino tilings, giving a recursive formula and generating function for one family, and describing a family of minimal-perimeter regions. 
\end{abstract}

\maketitle

\section{Introduction}\label{sec:intro}

Isoperimetric problems involve relationships between area and perimeter, and their study has a long history in geometry and analysis (see, for example, \cite{lusternik, stein shakarchi, steiner}). One can phrase this type of problem from two vantage points: fix the area and try to minimize perimeter, or fix the perimeter and try to maximize area. The ``answer'' to the classical \emph{contour-based} version of this problem, which was long suspected and has been known for over a century, is a circle.

\begin{theorem}[Classical contour-based isoperimetric theorem]\label{thm:classical isoperimetry}
Let $\Omega$ be a region bounded by a simple closed curve in the Euclidean plane. The Lebesgue measures for the length $P(\Omega)$ of that closed curve and the area $A(\Omega)$ of its enclosure $\Omega$ are related by
$$4\pi A(\Omega) \le P(\Omega)^2,$$
with equality if and only if the curve is a circle and $\Omega$ is a disk.
\end{theorem}

Discrete versions of isoperimetric problems exist in the literature, particularly within the context of graph theory (see, for example, \cite{chung, harper, woess}). Another discrete-flavored field of isoperimetric research concerns least-perimeter (under the Lebesgue measure) tilings of the plane, subject to possible tile restrictions, as discussed in \cite{chung et al, hales}.

In recent work with Duchin, we studied a frequent application of the isoperimetric problem to political science, and analyzed important flaws in that usage \cite{duchin-tenner}. In legal and political science literature, isoperimetric quotients of geographic regions are computed using contour-based measures of area and perimeter. These scores, and their relative rankings, have legal ramifications and electoral impact. However, contour-based calculations rely inherently on the way regions and curves are drawn on a surface, whereas the political redistricting problem begins, implicitly, with a graph of census units and their geographical/graphical adjacency. Any notion of ``area'' or ``perimeter'' for a collection of those units (such as those forming a district) should acknowledge this underlying structure. 

This paper takes the discretization of isoperimetric questions in a tiling-based direction. We turn our attention to polygonal tilings of bounded regions, and study the relationship between natural notions of ``area'' and ``perimeter'' in this context. We will use rhombic tilings of so-called ``Elnitsky polygons'' as our primary mechanism for this analysis. Such tilings of this family of polygons have algebraic significance, and the area and perimeter metrics that we develop for them will have significance as well. Another major area of tiling research concerns domino tilings, and we conclude this paper by turning our isoperimetric attention in that direction.

In Section~\ref{sec:background} of this work, we present motivating material about the symmetric group, and Corollary~\ref{cor:meaning of perimeter tiles} gives algebraic context to the notion of perimeter tiles. In Section~\ref{sec:long element}, we look at tiling-based perimeter questions for regular $2n$-gons, both in the extreme and average cases. We give sharp upper and lower bounds for the number of perimeter tiles in rhombic tilings of these regions in Theorems~\ref{thm:lower bound for total perimeter tiles} and~\ref{thm:upper bound for total perimeter tiles}, and present the Coxeter-theoretic significance of these results in Corollary~\ref{cor:max/min implications for commutation classes}. Section~\ref{sec:general perm} turns these isoperimetric questions to arbitrary permutations, and the main results of that section are Theorems~\ref{thm:exactly two perim tiles} and~\ref{thm:exactly two side-perim tiles} where we classify the permutations that can only ever have two perimeter or two side-perimeter tiles, respectively. We turn briefly to domino tilings in Section~\ref{sec:domino}, computing extreme and average values (Propositions~\ref{prop:2xn domino bounds} and~\ref{prop:domino total}) for rectangles. We also look at perimeter properties for non-rectangular regions, including a description of regions with minimally many perimeter tiles. We conclude the paper with a discussion of possible topics for future research.

\section{Background and motivation}\label{sec:background}

In this section, we introduce the objects and terminology related to the combinatorics of reduced decompositions of permutations. The reader is referred to \cite{bjorner-brenti} for more details.

Coxeter groups are generated by simple reflections. Minimally long expressions of a Coxeter group element in terms of these simple reflections are \emph{reduced decompositions} of that element. For a group element $w$, we write $R(w)$ for the collection of these reduced decompositions. The set $R(w)$ can be partitioned by the \emph{commutation} relation $\sim$, identifying reduced decompositions when they differ only by a sequence of commutation moves. The result is the collection $C(w) = R(w)/\!\sim$ of \emph{commutation classes} of $w$.

This paper is concerned with the finite Coxeter group of type $A$: the \emph{symmetric group}.

\begin{definition}
Let $\mathfrak{S}_n$ be the symmetric group on $[n]:=\{1,\ldots,n\}$. The elements of $\mathfrak{S}_n$ are \emph{permutations}, and we can write a permutation $w$ in \emph{one-line notation} as $w(1)w(2)\cdots w(n)$. For $i \in [n-1]$, let $s_i$ be the simple reflection transposing $i$ and $i+1$. Every permutation can be written as a product of simple reflections. The minimal length $\ell(w)$ of a product of simple reflections needed to represent $w$ is the \emph{length} of $w$, and this is equal to the number of \emph{inversions} in $w$: $\ell(w) = \# \{i < j : w(i) > w(j)\}$.
\end{definition}

Permutations are functions, and we compose them via: $w s_i$ transposes the values in positions $i$ and $i+1$ in $w$, while $s_i w$ transposes the positions of the values $i$ and $i+1$ in $w$.

We will use $42153 \in \mathfrak{S}_5$ as an example throughout this section.

\begin{example}\label{ex:42153 reduced decompositions}
The permutation $42153$ has eleven reduced decompositions
\begin{align*}
R(42153) = \{&s_1s_3s_2s_1s_4,\ s_1s_3s_2s_4s_1,\ s_1s_3s_4s_2s_1,\ s_3s_1s_2s_1s_4,\ s_3s_1s_2s_4s_1,\ s_3s_1s_4s_2s_1,\\ 
&s_3s_2s_1s_2s_4,\ s_3s_2s_1s_4s_2,\ s_3s_2s_4s_1s_2,\ s_3s_4s_1s_2s_1,\ s_3s_4s_2s_1s_2\},
\end{align*}
and two commutation classes
\begin{align*}
C(421&53) = \big\{
\{s_3s_2s_1s_2s_4,\ s_3s_2s_1s_4s_2,\ s_3s_2s_4s_1s_2,\ s_3s_4s_2s_1s_2\},\\
&\{s_1s_3s_2s_1s_4,\ s_1s_3s_2s_4s_1,\ s_1s_3s_4s_2s_1,\ s_3s_1s_2s_1s_4,\ s_3s_1s_2s_4s_1,\ s_3s_1s_4s_2s_1,\ s_3s_4s_1s_2s_1\}
\big\}.
\end{align*}
\end{example}

Reduced decompositions, commutation classes, and related objects are of interest from both algebraic and enumerative perspectives (see, for example, \cite{bedard, bergeron-ceballos-labbe, bjs, elnitsky, fishel-milicevic-patrias-tenner, jonsson-welker, meng, stanley, tenner patt-bru, tenner rdpp, tenner rwm, zollinger}). There is a wealth of interesting mathematics in these objects, and many open questions. In \cite{elnitsky}, Elnitsky gave a bijection between the commutation classes $C(w)$ and the rhombic tilings of a polygon $X(w)$. We explored that relationship further in \cite{tenner rdpp, tenner rwm}. Elnitsky's bijection provides important context for the work of this paper.

\begin{definition}\label{defn:elnitsky's polygon}
For $w \in \mathfrak{S}_n$, the polygon $X(w)$ is an equilateral $2n$-gon defined as follows: starting at the topmost vertex, label the sides $1,\ldots,n,w(n),\ldots,w(1)$ in counterclockwise order; the first (leftmost) $n$ of these sides (labeled $1,\ldots,n$) form half of a convex $2n$-gon; the remaining $n$ sides are drawn so that sides are parallel if and only if they have the same label. This $X(w)$ is \emph{Elnitsky's polygon} for $w$. A polygon obtained in this way for some permutation is an \emph{Elnitsky polygon}.
\end{definition}

\begin{example}
Elnitsky's polygon $X(42153)$ appears in Figure~\ref{fig:X(42153)}.

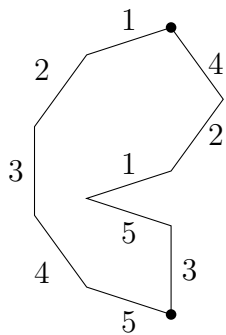
\begin{figure}[htbp]
\begin{tikzpicture}
\begin{scope}
\clip (.2,2) rectangle (-2,-2); 
\node[draw=none,minimum size=4cm,regular polygon,regular polygon sides=10] (a) {};
\foreach \x in {(a.side 1),(a.side 6)} {\fill \x circle (2pt);}
\end{scope}
\foreach \x in {1,2,3,4,5,6} {\coordinate (corner \x) at (a.side \x);}
\foreach \y [evaluate=\y as \x using \y+1] in {1,2,3,4,5} {\coordinate (side \y) at ($(corner \x)-(corner \y)$);}
\foreach \y [evaluate=\y as \x using \y+1] in {1,2,3,4,5} {\coordinate (side -\y) at ($(corner \y)-(corner \x)$);}
\draw (corner 1) -- (corner 2) -- (corner 3) -- (corner 4) -- (corner 5) -- (corner 6);
\draw (corner 1) --++ (side 4) coordinate (corner 10) --++ (side 2) coordinate (corner 9) --++ (side 1) coordinate (corner 8) --++ (side 5) coordinate (corner 7) --++ (side 3);
\draw ($(corner 1)!0.5!(corner 2)$) node[above] {$1$};
\draw ($(corner 2)!0.5!(corner 3)$) node[above left] {$2$};
\draw ($(corner 3)!0.5!(corner 4)$) node[left] {$3$};
\draw ($(corner 4)!0.5!(corner 5)$) node[below left] {$4$};
\draw ($(corner 5)!0.5!(corner 6)$) node[below] {$5$};
\draw ($(corner 6)!0.5!(corner 7)$) node[right] {$3$};
\draw ($(corner 7)!0.5!(corner 8)$) node[below] {$5$};
\draw ($(corner 8)!0.5!(corner 9)$) node[above] {$1$};
\draw ($(corner 9)!0.5!(corner 10)$) node[right] {$2$};
\draw ($(corner 10)!0.5!(corner 1)$) node[right] {$4$};
\end{tikzpicture}
\caption{Elnitsky's polygon $X(42153)$.}
\label{fig:X(42153)}
\end{figure}
\end{example}

If $\{w(1),\ldots,w(r)\} = \{1,\ldots, r\}$ for some $r < n$, then the interior of $X(w)$ is not contiguous. Here, we wish to ensure nonempty and contiguous polygonal area.

\begin{requirement}\label{req:contiguity}
Throughout this paper, if $w \in \mathfrak{S}_n$, then
\begin{enumerate}\renewcommand{\labelenumi}{(\alph{enumi})}
\item $n > 1$, and
\item $\{w(1),\ldots,w(r)\} \neq \{1,\ldots, r\}$ for all $r < n$.
\end{enumerate}
\end{requirement}

Definition~\ref{defn:elnitsky's polygon} does not specify angles beyond requiring convexity for the left half of the region. The requirement that $X(w)$ be equilateral could be replaced by requiring that same-labeled sides be congruent (as well as parallel). This is a superficial distinction and has no impact on the mathematics. The infinitely many ways to draw $X(w)$ are combinatorially equivalent, and we will refer to them as ``$X(w)$'' without distinction. We may omit edge labels in $X(w)$, but will compensate by marking the top and bottom vertices of the polygon.

\begin{definition}
The counterclockwise path from the top vertex to the bottom vertex in an Elnitsky polygon is the \emph{leftside} boundary of the polygon, and the clockwise path from the top vertex to the bottom vertex is the \emph{rightside} boundary. Note that the top and bottom vertices are in both the leftside and rightside boundaries of $X(w)$.
\end{definition}

The ``left'' and ``right'' sides of a tile are consistent with those of the surrounding polygon.

Elnitsky's work gives a bijection between commutation classes of a permutation and rhombic tilings of these polygons.

\begin{definition}
For a permutation $w$, a \emph{rhombic tiling} of $X(w)$ is a tiling in which all tile edges are congruent and parallel to edges of $X(w)$. The set of rhombic tilings of $X(w)$ is denoted $T(w)$.
\end{definition}

By ``path'' in a tiling, we will mean a path along tile edges. For $w \in \mathfrak{S}_n$ and a tiling in $T(w)$, any shortest path from the top vertex of $X(w)$ to the bottom will have exactly $n$ edges. This gives a sense of a tile's vertical position.

\begin{definition}
Fix $w$ and $T \in T(w)$. Let $t$ be a tile in $T$. If the topmost vertex of $t$ is $d$ edges from the top vertex of $X(w)$ in any shortest path, then the tile $t$ has \emph{depth} $d+1$.
\end{definition}

We are now ready to state Elnitsky's bijection.

\begin{theorem}[{\cite[Theorem 2.2]{elnitsky}}]\label{thm:elnitsky}
Fix $w \in \mathfrak{S}_n$ and $T \in T(w)$. There are $\ell(w)$ tiles in $T$. Label the tiles $1, 2, \ldots, \ell(w)$ so that the rightmost edges of a tile $t$ are shared with the rightside boundary of $X(w)$ and/or with tiles whose labels are less than the label of $t$. (Equivalently, the rightmost edges of a tile $t$ are not shared with any tiles whose labels are greater than the label of $t$.) This labeling corresponds to a reduced decomposition
$$s_{i_{\ell(w)}} \cdots s_{i_2} s_{i_1} \in R(w),$$
where $i_a = d$ if the tile labeled $a$ has depth $d$. The map from validly labeled tilings to reduced decompositions is a bijection. Moreover, two reduced decompositions correspond to different labelings of the same (unlabeled) tiling if and only if they belong to the same commutation class. In this way, a rhombic tiling of $X(w)$ produces an entire (and unique) commutation class of $R(w)$. 
\end{theorem}

\begin{example}\label{ex:T(42153)}
There are two rhombic tilings of $X(42153)$. The tiling depicted in Figure~\ref{fig:T(42153)}(a) corresponds to the first commutation class in Example~\ref{ex:42153 reduced decompositions}, and the tiling depicted in Figure~\ref{fig:T(42153)}(b) corresponds to the second. The four decompositions in the class
$$\{s_3s_2s_1s_2s_4, s_3s_2s_1s_4s_2, s_3s_2s_4s_1s_2, s_3s_4s_2s_1s_2\}$$
correspond to the four labelings of the tiling in Figure~\ref{fig:T(42153)}(a), depicted, respectively, in Figure~\ref{fig:T(42153) labelings}.
\end{example}

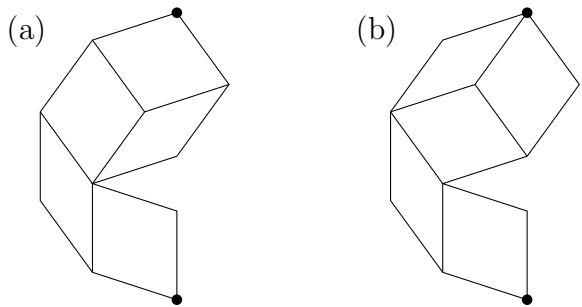
\begin{figure}[htbp]
\begin{tikzpicture}
\begin{scope}
\clip (.2,2) rectangle (-2,-2);
\node[draw=none,minimum size=4cm,regular polygon,regular polygon sides=10] (a) {};
\foreach \x in {(a.side 1),(a.side 6)} {\fill \x circle (2pt);}
\end{scope}
\foreach \x in {1,2,3,4,5,6} {\coordinate (corner \x) at (a.side \x);}
\foreach \y [evaluate=\y as \x using \y+1] in {1,2,3,4,5} {\coordinate (side \y) at ($(corner \x)-(corner \y)$);}
\foreach \y [evaluate=\y as \x using \y+1] in {1,2,3,4,5} {\coordinate (side -\y) at ($(corner \y)-(corner \x)$);}
\draw (corner 1) -- (corner 2) -- (corner 3) -- (corner 4) -- (corner 5) -- (corner 6);
\draw (corner 1) --++ (side 4) coordinate (corner 10) --++ (side 2) coordinate (corner 9) --++ (side 1) coordinate (corner 8) --++ (side 5) coordinate (corner 7) --++ (side 3);
\draw (corner 5) -- (corner 8) -- (corner 3);
\draw (corner 2) --++ (side 4) --++ (side -1);
\draw (corner 8) --++ (side -2);
\draw ($(corner 1) + (-2,-.25)$) node {(a)};
\end{tikzpicture}
\hspace{.5in}
\begin{tikzpicture}
\begin{scope}
\clip (.2,2) rectangle (-2,-2);
\node[draw=none,minimum size=4cm,regular polygon,regular polygon sides=10] (a) {};
\foreach \x in {(a.side 1),(a.side 6)} {\fill \x circle (2pt);}
\end{scope}
\foreach \x in {1,2,3,4,5,6} {\coordinate (corner \x) at (a.side \x);}
\foreach \y [evaluate=\y as \x using \y+1] in {1,2,3,4,5} {\coordinate (side \y) at ($(corner \x)-(corner \y)$);}
\foreach \y [evaluate=\y as \x using \y+1] in {1,2,3,4,5} {\coordinate (side -\y) at ($(corner \y)-(corner \x)$);}
\draw (corner 1) -- (corner 2) -- (corner 3) -- (corner 4) -- (corner 5) -- (corner 6);
\draw (corner 1) --++ (side 4) coordinate (corner 10) --++ (side 2) coordinate (corner 9) --++ (side 1) coordinate (corner 8) --++ (side 5) coordinate (corner 7) --++ (side 3);
\draw (corner 5) -- (corner 8) -- (corner 3);
\draw (corner 1) --++ (side 2) --++ (side 1);
\draw (corner 9) --++ (side -4);
\draw ($(corner 1) + (-2,-.25)$) node {(b)};
\end{tikzpicture}
\caption{The rhombic tilings $T(42153)$, discussed in Example~\ref{ex:T(42153)}.}
\label{fig:T(42153)}
\end{figure}

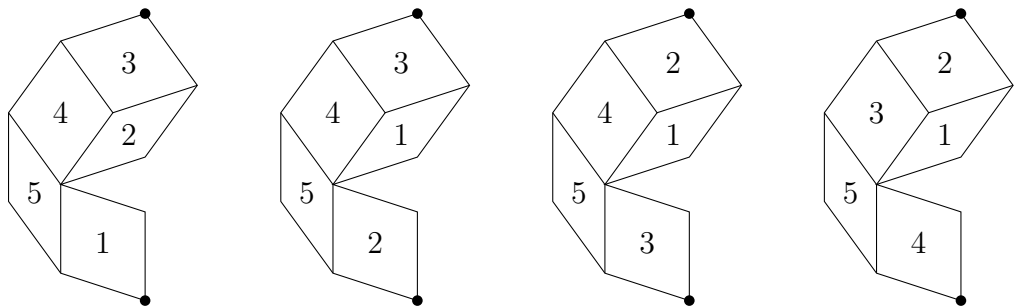
\begin{figure}[htbp]
\begin{tikzpicture}
\begin{scope}
\clip (.2,2) rectangle (-2,-2);
\node[draw=none,minimum size=4cm,regular polygon,regular polygon sides=10] (a) {};
\foreach \x in {(a.side 1),(a.side 6)} {\fill \x circle (2pt);}
\end{scope}
\foreach \x in {1,2,3,4,5,6} {\coordinate (corner \x) at (a.side \x);}
\foreach \y [evaluate=\y as \x using \y+1] in {1,2,3,4,5} {\coordinate (side \y) at ($(corner \x)-(corner \y)$);}
\foreach \y [evaluate=\y as \x using \y+1] in {1,2,3,4,5} {\coordinate (side -\y) at ($(corner \y)-(corner \x)$);}
\draw (corner 1) -- (corner 2) -- (corner 3) -- (corner 4) -- (corner 5) -- (corner 6);
\draw (corner 1) --++ (side 4) coordinate (corner 10) --++ (side 2) coordinate (corner 9) --++ (side 1) coordinate (corner 8) --++ (side 5) coordinate (corner 7) --++ (side 3);
\draw (corner 5) -- (corner 8) -- (corner 3);
\draw (corner 2) --++ (side 4) --++ (side -1);
\draw (corner 8) --++ (side -2) coordinate (corner A);
\draw ($(corner 1) + (-2,-.25)$);
\draw ($(corner 9)!0.5!(corner A)$) node {$2$};
\draw ($(corner 1)!0.5!(corner A)$) node {$3$};
\draw ($(corner 3)!0.5!(corner A)$) node {$4$};
\draw ($(corner 5)!0.5!(corner 7)$) node {$1$};
\draw ($(corner 4)!0.5!(corner 8)$) node {$5$};
\end{tikzpicture}
\hspace{.25in}
\begin{tikzpicture}
\begin{scope}
\clip (.2,2) rectangle (-2,-2);
\node[draw=none,minimum size=4cm,regular polygon,regular polygon sides=10] (a) {};
\foreach \x in {(a.side 1),(a.side 6)} {\fill \x circle (2pt);}
\end{scope}
\foreach \x in {1,2,3,4,5,6} {\coordinate (corner \x) at (a.side \x);}
\foreach \y [evaluate=\y as \x using \y+1] in {1,2,3,4,5} {\coordinate (side \y) at ($(corner \x)-(corner \y)$);}
\foreach \y [evaluate=\y as \x using \y+1] in {1,2,3,4,5} {\coordinate (side -\y) at ($(corner \y)-(corner \x)$);}
\draw (corner 1) -- (corner 2) -- (corner 3) -- (corner 4) -- (corner 5) -- (corner 6);
\draw (corner 1) --++ (side 4) coordinate (corner 10) --++ (side 2) coordinate (corner 9) --++ (side 1) coordinate (corner 8) --++ (side 5) coordinate (corner 7) --++ (side 3);
\draw (corner 5) -- (corner 8) -- (corner 3);
\draw (corner 2) --++ (side 4) --++ (side -1);
\draw (corner 8) --++ (side -2) coordinate (corner A);
\draw ($(corner 1) + (-2,-.25)$);
\draw ($(corner 9)!0.5!(corner A)$) node {$1$};
\draw ($(corner 1)!0.5!(corner A)$) node {$3$};
\draw ($(corner 3)!0.5!(corner A)$) node {$4$};
\draw ($(corner 5)!0.5!(corner 7)$) node {$2$};
\draw ($(corner 4)!0.5!(corner 8)$) node {$5$};
\end{tikzpicture}
\hspace{.25in}
\begin{tikzpicture}
\begin{scope}
\clip (.2,2) rectangle (-2,-2);
\node[draw=none,minimum size=4cm,regular polygon,regular polygon sides=10] (a) {};
\foreach \x in {(a.side 1),(a.side 6)} {\fill \x circle (2pt);}
\end{scope}
\foreach \x in {1,2,3,4,5,6} {\coordinate (corner \x) at (a.side \x);}
\foreach \y [evaluate=\y as \x using \y+1] in {1,2,3,4,5} {\coordinate (side \y) at ($(corner \x)-(corner \y)$);}
\foreach \y [evaluate=\y as \x using \y+1] in {1,2,3,4,5} {\coordinate (side -\y) at ($(corner \y)-(corner \x)$);}
\draw (corner 1) -- (corner 2) -- (corner 3) -- (corner 4) -- (corner 5) -- (corner 6);
\draw (corner 1) --++ (side 4) coordinate (corner 10) --++ (side 2) coordinate (corner 9) --++ (side 1) coordinate (corner 8) --++ (side 5) coordinate (corner 7) --++ (side 3);
\draw (corner 5) -- (corner 8) -- (corner 3);
\draw (corner 2) --++ (side 4) --++ (side -1);
\draw (corner 8) --++ (side -2) coordinate (corner A);
\draw ($(corner 1) + (-2,-.25)$);
\draw ($(corner 9)!0.5!(corner A)$) node {$1$};
\draw ($(corner 1)!0.5!(corner A)$) node {$2$};
\draw ($(corner 3)!0.5!(corner A)$) node {$4$};
\draw ($(corner 5)!0.5!(corner 7)$) node {$3$};
\draw ($(corner 4)!0.5!(corner 8)$) node {$5$};
\end{tikzpicture}
\hspace{.25in}
\begin{tikzpicture}
\begin{scope}
\clip (.2,2) rectangle (-2,-2);
\node[draw=none,minimum size=4cm,regular polygon,regular polygon sides=10] (a) {};
\foreach \x in {(a.side 1),(a.side 6)} {\fill \x circle (2pt);}
\end{scope}
\foreach \x in {1,2,3,4,5,6} {\coordinate (corner \x) at (a.side \x);}
\foreach \y [evaluate=\y as \x using \y+1] in {1,2,3,4,5} {\coordinate (side \y) at ($(corner \x)-(corner \y)$);}
\foreach \y [evaluate=\y as \x using \y+1] in {1,2,3,4,5} {\coordinate (side -\y) at ($(corner \y)-(corner \x)$);}
\draw (corner 1) -- (corner 2) -- (corner 3) -- (corner 4) -- (corner 5) -- (corner 6);
\draw (corner 1) --++ (side 4) coordinate (corner 10) --++ (side 2) coordinate (corner 9) --++ (side 1) coordinate (corner 8) --++ (side 5) coordinate (corner 7) --++ (side 3);
\clip (corner 1) -- (corner 2) -- (corner 3) -- (corner 4) -- (corner 5) -- (corner 6) -- (corner 7) -- (corner 8) -- (corner 9) -- (corner 10) -- (corner 1);
\draw (corner 5) -- (corner 8) -- (corner 3);
\draw (corner 2) --++ (side 4) --++ (side -1);
\draw (corner 8) --++ (side -2) coordinate (corner A);
\draw ($(corner 1) + (-2,-.25)$);
\draw ($(corner 9)!0.5!(corner A)$) node {$1$};
\draw ($(corner 1)!0.5!(corner A)$) node {$2$};
\draw ($(corner 3)!0.5!(corner A)$) node {$3$};
\draw ($(corner 5)!0.5!(corner 7)$) node {$4$};
\draw ($(corner 4)!0.5!(corner 8)$) node {$5$};
\end{tikzpicture}
\caption{The four permitted labelings of the tiling in Figure~\ref{fig:T(42153)}(a), corresponding to the four elements in one of the commutation classes of $R(42153)$.}
\label{fig:T(42153) labelings}
\end{figure}

The definition of Elnitsky's polygon and its rhombic tilings has additional implications for what sorts of tiles may appear in elements of $T(w)$.

\begin{definition}
If the edges of a tile $t$ in $T \in T(w)$ are parallel to the sides labeled $a$ and $b$ in $X(w)$, then the \emph{edge labels} of $t$ are $\{a,b\}$.
\end{definition}
 
In Elnitsky's bijection, each tile corresponds to an inversion in $w$ \cite{elnitsky}. Moreover, if $i<j$ and $w(i) > w(j)$, then each $T \in T(w)$ contains a tile whose edge labels are $\{w(i),w(j)\}$.

\begin{corollary}\label{cor:rhombi can't have same labels}
For any $w$ and $T \in T(w)$, no two tiles in $T$ have the same edge labels.
\end{corollary}

\begin{proof}
For each $\{w(i)>w(j)\}$ with $i<j$, there is a tile in $T$ with edge labels $\{w(i),w(j)\}$. The number of tiles in $T$ is $\ell(w)$, which is equal to the number of such $(i,j)$. Thus each such pair of edge labels appears exactly once and no other pairs of edge labels appear at all.
\end{proof}

Consider a simple closed curve and the region that it encloses. If that region is tiled (partitioned) by some collection of shapes, then we can identify the tiles that intersect the boundary nontrivially. Suppose that the curve, region, and tiling are \emph{discrete} in the sense that all shapes (region and tiles) are polygons, and pairs of tile edges whose intersection contains more than a single point actually intersect in their entirety. Then there is an intuitive tiling-based distance metric in this setting: each tile edge has ``discrete length'' one, and the ``discrete length'' of a path is the number of tile edges that the path contains. Similarly, there is an intuitive notion of ``discrete area,'' as defined below.

\begin{definition}\label{defn:tiling-based area}
Let $R$ be a region with polygonal boundary, and $T$ a discrete tiling of $R$, in the sense defined above. The \emph{area} of this tiling of $R$ is the number of tiles in $T$.
\end{definition}

The tilings we will discuss have the property that every tiling of a region $R$ has the same number of tiles, and so the area depends on the region $R$ alone. Because the regions and tiles that we study are all polygons, the interesting question about tiling behavior along the boundary of $R$ is: how many tiles share multiple edges with the boundary of $R$? That is, conspicuous tile/boundary intersection occurs when there is prolonged (multi-edge) overlap.

\begin{definition}\label{defn:perimeter tile generally}
Let $R$ be a region with polygonal boundary, and $T$ a discrete tiling of $R$. A tile in $T$ that shares a path of discrete length at least two with the boundary of $R$ is a \emph{strong perimeter tile}, or simply a \emph{perimeter tile}. A non-strong perimeter tile in $T$ that shares a path of positive discrete length with the boundary of $R$ is a \emph{weak perimeter tile}.
\end{definition}

For a permutation $w$ and a rhombic tiling $T \in T(w)$, let
$$\perim(T)$$
be the number of (strong) perimeter tiles in $T$. The total number of weak and strong perimeter tiles in a tiling $T \in T(w)$ (that is, the number of tiles sharing paths of positive length with the boundary of $X(w)$) is less than or equal to
$$2n - \perim(T),$$
depending on whether any weak or strong perimeter tiles in $T$ share more than the minimally required number of edges with the boundary of $X(w)$. There is certainly a case to be made for studying this quantity. However, as it depends so closely on $\perim$, and because of the relevance of (strong) perimeter tiles discussed in the rest of this section, we focus our attention on the statistic $\perim$ in this paper.

\begin{definition}\label{defn:perimeter tiles}
Fix a permutation $w$ and consider a rhombic tiling of Elnitsky's polygon $X(w)$. If, in such a tiling, a perimeter tile $t$ includes $\ldots$
\begin{itemize}
\item two edges from the leftside boundary of $X(w)$, then $t$ is a \emph{left-perimeter tile};
\item two edges from the rightside boundary of $X(w)$, then $t$ is a \emph{right-perimeter tile};
\item the two boundary edges to the left and right of the top vertex of $X(w)$, then $t$ is a \emph{top-perimeter tile};
\item the two boundary edges to the left and right of the bottom vertex of $X(w)$, then $t$ is a \emph{bottom-perimeter tile}.
\end{itemize}
The \emph{type} of a perimeter tile is its classification as left-, right-, top-, and/or bottom-. \emph{Side-perimeter} will refer to tiles whose types are left- or right-.
\end{definition}

\begin{example}
Consider the tiling of $X(42153)$ shown in Figure~\ref{fig:T(42153)}(a), with tiles labeled as in leftmost picture in Figure~\ref{fig:T(42153) labelings}. This tiling has one left-perimeter tile (tile $5$), two right-perimeter tiles (tiles $1$ and $2$), one top-perimeter tile (tile $3$), and one bottom-perimeter tile (tile $1$). The bottom-perimeter tile is also a right-perimeter tile.
\end{example}

The significance of perimeter tiles follows directly from Elnitsky's bijection (Theorem~\ref{thm:elnitsky}), because one can say precisely which labeled tiling corresponds to each reduced decomposition.

\begin{corollary}\label{cor:meaning of perimeter tiles}
Fix $w\in \mathfrak{S}_n$. Some $T \in T(w)$ has a $\ldots$
\begin{enumerate}
\item[(a)] left-perimeter tile of depth $d$ if and only if $\ell(s_d w) < \ell(w)$; equivalently, there exists a reduced decomposition of $w$ in which $s_d$ is the leftmost letter in the product.
\item[(b)] right-perimeter tile of depth $d$ if and only if $\ell(w s_d) < \ell(w)$; equivalently, there exists a reduced decompositions of $w$ in which $s_d$ is the rightmost letter in the product.
\item[(c)] top-perimeter tile if and only if there exists a commutation class $C \in C(w)$ in which there is exactly one $s_1$ in each reduced decomposition in $C$.
\item[(d)] bottom-perimeter tile if and only if there exists a commutation class $C \in C(w)$ in which there is exactly one $s_{n-1}$ in each reduced decomposition in $C$.
\end{enumerate}
\end{corollary}

\begin{example}
In the tilings of $X(42153)$ (see Figure~\ref{fig:T(42153)}), there are left-perimeter tiles of depths $1$ and $3$, right-perimeter tiles of depths $1$, $2$, and $4$, top-perimeter tiles, and bottom-perimeter tiles. The left-perimeter tiles indicate that the only leftmost letters that appear in elements of $R(42153)$ are $s_1$ and $s_3$. The right-perimeter tiles indicate that the only rightmost letters that appear are $s_1$, $s_2$, and $s_4$. The top-perimeter tile indicates that there is a commutation class in which all reduced decompositions contain exactly one $s_1$, and the bottom-perimeter tiles indicate that there are two commutation classes in which all reduced decompositions contain exactly one $s_4$. Moreover, because one tiling of $X(42153)$ has both a top- and a bottom-perimeter tile, there is a commutation class in which each element contains exactly one $s_1$ and exactly one $s_4$. These conclusions agree with Example~\ref{ex:42153 reduced decompositions}.
\end{example}

It is clear that the perimeter tiles in elements of $T(w)$ relate to the combinatorics of reduced decompositions of permutations. The purpose of this paper is to understand when, how, and how often these perimeter tiles occur. We start with an easy observation about an upper bound for the number of perimeter tiles that can occur in any tiling. As we will see in Theorem~\ref{thm:upper bound for total perimeter tiles}, this can sometimes be tightened.

\begin{proposition}\label{prop:at most n perimeter rhombi}
For any $w \in \mathfrak{S}_n$, each tiling in $T(w)$ has at most $n$ perimeter tiles.
\end{proposition}

\begin{proof}
The polygon $X(w)$ has $2n$ edges. Each perimeter tile accounts for at least two of these edges, and no edge of $X(w)$ can belong to more than one tile.
\end{proof}

For an initial lower bound on the number of perimeter tiles that can appear in a tiling, we draw an immediate consequence from Elnitsky's theorem and Corollary~\ref{cor:meaning of perimeter tiles}.

\begin{corollary}\label{cor:always have left and right perimeter tiles}
For any $w$ and $T \in T(w)$, there is at least one left-perimeter tile and at least one right-perimeter tile in $T$.
\end{corollary}

Corollary~\ref{cor:always have left and right perimeter tiles} does not require that those left- and right-perimeter tiles be distinct. However, it is clear from Theorem~\ref{thm:elnitsky} and Requirement~\ref{req:contiguity}(b) that the only situation in which distinct left- and right-perimeter tiles cannot be found is when $w = 21 \in \mathfrak{S}_2$.

\section{Perimeter tiles for the longest element}\label{sec:long element} 

We devote this section to understanding perimeter tiles that appear in $X(\longelt{n})$, where
$$\longelt{n} = n(n-1)\cdots 321$$
is the longest element in $\mathfrak{S}_n$. When $n$ is clear from context, we may write $w_0 := \longelt{n}$. Our primary focus is on extremal data, with a brief discussion of averages at the end of the section. Recall Requirement~\ref{req:contiguity}(a), that $n > 1$. The element $\longelt{n}$ and its properties are of particular Coxeter-theoretic interest, and have been studied often (see, for example, \cite{reiner, tenner comm} and the works cited previously).

\begin{definition}\label{defn:perimeter counting functions}
For $w$ and $T \in T(w)$, recall that $\perim(T)$ is the number of (strong) perimeter tiles in $T$. When tile type is relevant, the functions $\leftperim$, $\rightperim$, $\topperim$, and $\bottomperim$ will be used in the obvious way.
We say that the value of $\perim(T)$ measures the \emph{strong perimeter} of the tiling $T$. If all $T \in T(w)$ have the same strong perimeter, then this value is the \emph{strong perimeter} of $w$.
\end{definition}

For any $w \in \mathfrak{S}_n$ and $T \in T(w)$, $\leftperim(T)$ and $\rightperim(T)$ are integers between $0$ and $n/2$, while $\topperim(T)$ and $\bottomperim(T)$ are indicator variables taking values $0$ or $1$.

Our intention is to understand the statistic $\perim(T)$, particularly when $w = \longelt{n}$. Because $X(\longelt{n})$ is a convex centrally symmetric $2n$-gon, we can make use of its symmetries in our arguments, and we start this section with an elementary, but important, observation about regular polygons. It simplifies the argument, and does not change the result, to assume, for the moment, that $X(w_0)$ is equiangular. 

\begin{lemma}\label{lem:can rotate}
Consider a regular $2n$-gon $X$ and the set $T(X)$ of rhombic tilings of $X$, in which the sides of each rhombus are congruent and parallel to sides of $X$. For any $T \in T(X)$, let $T'$ be the result of rotating $T$ about the center of $X$ by an integer multiple of $(2\pi)/(2n)$ radians. Then $T' \in T(X)$.
\end{lemma}

Although Elnitsky's polygon $X(w_0)$ need not be equiangular, the result of Lemma~\ref{lem:can rotate} applies: rhombic tilings of $X(w_0)$ can be rotated to produce other (not necessarily distinct) rhombic tilings of $X(w_0)$. An example of this is depicted in Figure~\ref{fig:can rotate tilings}.

\begin{figure}[htbp]
\begin{tikzpicture}
\node[draw=none,minimum size=4cm,regular polygon,regular polygon sides=12] (a) {};
\foreach \x in {(a.side 1),(a.side 7)} {\fill \x circle (2pt);}
\foreach \x in {1,...,12} {\coordinate (corner \x) at (a.side \x);}
\foreach \y [evaluate=\y as \x using \y+1] in {1,...,6} {\coordinate (side \y) at ($(corner \x)-(corner \y)$);}
\foreach \y [evaluate=\y as \x using \y+1] in {1,...,6} {\coordinate (side -\y) at ($(corner \y)-(corner \x)$);}
\foreach \y [evaluate=\y as \x using \y+1] in {1,...,11} {\draw (corner \y) -- (corner \x);}
\draw (corner 12) -- (corner 1);
\draw (corner 2) --++(side 3) coordinate (a) --++(side 2);
\draw (corner 5) --++(side 6) coordinate (d) --++(side 5);
\draw (corner 5) --++(side -2) --++(side 6) --++(side 2);
\draw (a) --++(side 4);
\draw (corner 2) --++(side 4) --++(side 3);
\draw (corner 1) --++(side 4) coordinate (b) --++(side 1);
\draw (b) --++(side 3) --++(side 1);
\draw (d) --++(side -1) --++(side -2) coordinate (e) --++(side 1);
\draw (e) --++(side -6);
\draw (corner 8) --++(side -5);
\draw (corner 9) --++(side -5);
\draw (corner 12) --++(side 4) coordinate (c)--++(side 5);
\draw (c) --++(side -6);
\draw (c) --++(side 3);
\end{tikzpicture}
\hspace{.5in}
\begin{tikzpicture}
\node[draw=none,minimum size=4cm,regular polygon,regular polygon sides=12] (a) {};
\foreach \x in {1,...,12} {\coordinate (fakecorner \x) at (a.side \x);}
\foreach \y [evaluate=\y as \x using \y+1] in {1,...,11} {\coordinate (corner \y) at (fakecorner \x);}
\coordinate (corner 12) at (fakecorner 1);
\fill (corner 12) circle (2pt);
\fill (corner 6) circle (2pt);
\foreach \y [evaluate=\y as \x using \y+1] in {1,...,6} {\coordinate (side \y) at ($(corner \x)-(corner \y)$);}
\foreach \y [evaluate=\y as \x using \y+1] in {1,...,6} {\coordinate (side -\y) at ($(corner \y)-(corner \x)$);}
\foreach \y [evaluate=\y as \x using \y+1] in {1,...,11} {\draw (corner \y) -- (corner \x);}
\draw (corner 12) -- (corner 1);
\draw (corner 12) -- (corner 1);
\draw (corner 2) --++(side 3) coordinate (a) --++(side 2);
\draw (corner 5) --++(side 6) coordinate (d) --++(side 5);
\draw (corner 5) --++(side -2) --++(side 6) --++(side 2);
\draw (a) --++(side 4);
\draw (corner 2) --++(side 4) --++(side 3);
\draw (corner 1) --++(side 4) coordinate (b) --++(side 1);
\draw (b) --++(side 3) --++(side 1);
\draw (d) --++(side -1) --++(side -2) coordinate (e) --++(side 1);
\draw (e) --++(side -6);
\draw (corner 8) --++(side -5);
\draw (corner 9) --++(side -5);
\draw (corner 12) --++(side 4) coordinate (c)--++(side 5);
\draw (c) --++(side -6);
\draw (c) --++(side 3);
\end{tikzpicture}
\caption{Two rhombic tilings of $X(\protect\longelt{6})$, which differ by a rotation.}
\label{fig:can rotate tilings}
\end{figure}
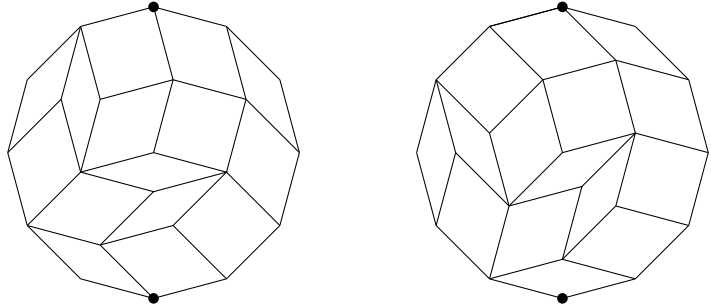

\begin{proposition}\label{prop:symmetry of x(w_0)}
Over all rhombic tilings of $X(\longelt{n})$, the following quantities are equal:
\begin{itemize}
\item the total number of top-perimeter tiles that appear,
\item the total number of bottom-perimeter tiles that appear,
\item the total number of left-perimeter tiles of depth $d$ that appear, for $d \in [n-1]$, and
\item the total number of right-perimeter tiles of depth $d$ that appear, for $d \in [n-1]$.
\end{itemize}
\end{proposition}

\begin{proof}
Suppose, without loss of generality, that $X(w_0)$ is equiangular. Then, by Lemma~\ref{lem:can rotate}, there is nothing special about a perimeter tile appearing at, say, depth $1$ along the rightside boundary of $X(w_0)$, and so no perimeter location is preferred to any other amongst all rhombic tilings of $X(w_0)$.
\end{proof}

Put another way, Proposition~\ref{prop:symmetry of x(w_0)} says that if $\textsf{L}_d\textsf{-perim}$ and $\textsf{R}_d\textsf{-perim}$ are indicator variables that detect depth $d$ left-, and depth $d$ right-perimeter tiles, respectively, then
$$\sum_{T \in T(w_0)} \topperim(T) = \sum_{T \in T(w_0)} \bottomperim(T) = \sum_{T \in T(w_0)} \textsf{L}_d\textsf{-perim}(T) = \sum_{T \in T(w_0)} \textsf{R}_d\textsf{-perim}(T),$$
for any $d \in [n-1]$. Thus the total number of each type of perimeter tile (at fixed depth, if the type is left- or right-) is dependent only on $n$, and not on the type itself.

\begin{example}\label{ex:equal appearances in rhombic tilings of X(longelt{4})}
Among the eight rhombic tilings of $X(\protect\longelt{4})$ (depicted in Figure~\ref{fig:rhombic tilings of X(longelt{4})}), there are three appearances each of top-perimeter tiles, bottom-perimeter tiles, left-perimeter tiles of any fixed depth, and right-perimeter tiles of any fixed depth. 
\end{example}

\begin{figure}[htbp]
\begin{tikzpicture}
\node[draw=none,minimum size=3cm,regular polygon,regular polygon sides=8] (a) {};
\foreach \x in {(a.side 1),(a.side 5)} {\fill \x circle (2pt);}
\foreach \x in {1,...,8} {\coordinate (corner \x) at (a.side \x);}
\foreach \y [evaluate=\y as \x using \y+1] in {1,...,4} {\coordinate (side \y) at ($(corner \x)-(corner \y)$);}
\foreach \y [evaluate=\y as \x using \y+1] in {1,...,4} {\coordinate (side -\y) at ($(corner \y)-(corner \x)$);}
\foreach \y [evaluate=\y as \x using \y+1] in {1,...,7} {\draw (corner \y) -- (corner \x);}
\draw (corner 8) -- (corner 1);
\foreach \y [evaluate=\y as \x using \y+1] in {2,...,3} {\draw (corner \x) --++(side -1) --++(side -\y);}
\foreach \y [evaluate=\y as \x using \y-5] in {6,...,8} {\draw (corner \y) --++(side -4) --++(side \x);}
\end{tikzpicture}
\hspace{.25in}
\begin{tikzpicture}
\node[draw=none,minimum size=3cm,regular polygon,regular polygon sides=8] (a) {};
\foreach \x in {(a.side 1),(a.side 5)} {\fill \x circle (2pt);}
\foreach \x in {1,...,8} {\coordinate (corner \x) at (a.side \x);}
\foreach \y [evaluate=\y as \x using \y+1] in {1,...,4} {\coordinate (side \y) at ($(corner \x)-(corner \y)$);}
\foreach \y [evaluate=\y as \x using \y+1] in {1,...,4} {\coordinate (side -\y) at ($(corner \y)-(corner \x)$);}
\foreach \y [evaluate=\y as \x using \y+1] in {1,...,7} {\draw (corner \y) -- (corner \x);}
\draw (corner 8) -- (corner 1);
\foreach \y [evaluate=\y as \x using \y+1] in {2,...,3} {\draw (corner \x) --++(side -1) --++(side -\y);}
\draw (corner 6) --++(side -4);
\draw (corner 8) --++(side 2) --++(side -4);
\draw (corner 6) --++(side -3);
\end{tikzpicture}
\hspace{.25in}
\begin{tikzpicture}
\node[draw=none,minimum size=3cm,regular polygon,regular polygon sides=8] (a) {};
\foreach \x in {(a.side 1),(a.side 5)} {\fill \x circle (2pt);}
\foreach \x in {1,...,8} {\coordinate (corner \x) at (a.side \x);}
\foreach \y [evaluate=\y as \x using \y+1] in {1,...,4} {\coordinate (side \y) at ($(corner \x)-(corner \y)$);}
\foreach \y [evaluate=\y as \x using \y+1] in {1,...,4} {\coordinate (side -\y) at ($(corner \y)-(corner \x)$);}
\foreach \y [evaluate=\y as \x using \y+1] in {1,...,7} {\draw (corner \y) -- (corner \x);}
\draw (corner 8) -- (corner 1);
\draw (corner 1) --++(side 2) --++(side 1);
\draw (corner 8) --++(side 2) --++(side -4);
\draw (corner 6) --++(side -3);
\draw (corner 5) --++(side -3) --++(side -1);
\draw (corner 3) --++(side 4);
\end{tikzpicture}
\hspace{.25in}
\begin{tikzpicture}
\node[draw=none,minimum size=3cm,regular polygon,regular polygon sides=8] (a) {};
\foreach \x in {(a.side 1),(a.side 5)} {\fill \x circle (2pt);}
\foreach \x in {1,...,8} {\coordinate (corner \x) at (a.side \x);}
\foreach \y [evaluate=\y as \x using \y+1] in {1,...,4} {\coordinate (side \y) at ($(corner \x)-(corner \y)$);}
\foreach \y [evaluate=\y as \x using \y+1] in {1,...,4} {\coordinate (side -\y) at ($(corner \y)-(corner \x)$);}
\foreach \y [evaluate=\y as \x using \y+1] in {1,...,7} {\draw (corner \y) -- (corner \x);}
\draw (corner 8) -- (corner 1);
\draw (corner 8) --++(side 2);
\draw (corner 6) --++(side -3);
\draw (corner 5) --++(side -3) --++(side -1);
\draw (corner 3) --++(side 4);
\draw (corner 2) --++(side 4) --++(side 2);
\draw (corner 8) --++(side 1);
\end{tikzpicture}

\vspace{.25in}

\begin{tikzpicture}
\node[draw=none,minimum size=3cm,regular polygon,regular polygon sides=8] (a) {};
\foreach \x in {(a.side 1),(a.side 5)} {\fill \x circle (2pt);}
\foreach \x in {1,...,8} {\coordinate (corner \x) at (a.side \x);}
\foreach \y [evaluate=\y as \x using \y+1] in {1,...,4} {\coordinate (side \y) at ($(corner \x)-(corner \y)$);}
\foreach \y [evaluate=\y as \x using \y+1] in {1,...,4} {\coordinate (side -\y) at ($(corner \y)-(corner \x)$);}
\foreach \y [evaluate=\y as \x using \y+1] in {1,...,7} {\draw (corner \y) -- (corner \x);}
\draw (corner 8) -- (corner 1);
\draw (corner 8) --++(side 1) --++(side 2) --++(side 3);
\draw (corner 2) --++(side 4);
\draw (corner 3) --++(side 4);
\draw (corner 5) --++(side -2) --++(side -3);
\draw (corner 7) --++(side 1);
\end{tikzpicture}
\hspace{.25in}
\begin{tikzpicture}
\node[draw=none,minimum size=3cm,regular polygon,regular polygon sides=8] (a) {};
\foreach \x in {(a.side 1),(a.side 5)} {\fill \x circle (2pt);}
\foreach \x in {1,...,8} {\coordinate (corner \x) at (a.side \x);}
\foreach \y [evaluate=\y as \x using \y+1] in {1,...,4} {\coordinate (side \y) at ($(corner \x)-(corner \y)$);}
\foreach \y [evaluate=\y as \x using \y+1] in {1,...,4} {\coordinate (side -\y) at ($(corner \y)-(corner \x)$);}
\foreach \y [evaluate=\y as \x using \y+1] in {1,...,7} {\draw (corner \y) -- (corner \x);}
\draw (corner 8) -- (corner 1);
\draw (corner 2) --++(side 4) --++(side 3) --++(side 2);
\draw (corner 7) --++(side 	1);
\draw (corner 8) --++(side 1);
\draw (corner 4) --++(side -2) --++(side 4);
\draw (corner 2) --++(side 3);
\end{tikzpicture}
\hspace{.25in}
\begin{tikzpicture}
\node[draw=none,minimum size=3cm,regular polygon,regular polygon sides=8] (a) {};
\foreach \x in {(a.side 1),(a.side 5)} {\fill \x circle (2pt);}
\foreach \x in {1,...,8} {\coordinate (corner \x) at (a.side \x);}
\foreach \y [evaluate=\y as \x using \y+1] in {1,...,4} {\coordinate (side \y) at ($(corner \x)-(corner \y)$);}
\foreach \y [evaluate=\y as \x using \y+1] in {1,...,4} {\coordinate (side -\y) at ($(corner \y)-(corner \x)$);}
\foreach \y [evaluate=\y as \x using \y+1] in {1,...,7} {\draw (corner \y) -- (corner \x);}
\draw (corner 8) -- (corner 1);
\draw (corner 2) --++(side 3) --++(side 4) --++(side -1);
\draw (corner 4) --++(side 	-2);
\draw (corner 5) --++(side -2);
\draw (corner 1) --++(side 3) --++(side 1);
\draw (corner 7) --++(side -4);
\end{tikzpicture}
\hspace{.25in}
\begin{tikzpicture}
\node[draw=none,minimum size=3cm,regular polygon,regular polygon sides=8] (a) {};
\foreach \x in {(a.side 1),(a.side 5)} {\fill \x circle (2pt);}
\foreach \x in {1,...,8} {\coordinate (corner \x) at (a.side \x);}
\foreach \y [evaluate=\y as \x using \y+1] in {1,...,4} {\coordinate (side \y) at ($(corner \x)-(corner \y)$);}
\foreach \y [evaluate=\y as \x using \y+1] in {1,...,4} {\coordinate (side -\y) at ($(corner \y)-(corner \x)$);}
\foreach \y [evaluate=\y as \x using \y+1] in {1,...,7} {\draw (corner \y) -- (corner \x);}
\draw (corner 8) -- (corner 1);
\draw (corner 4) --++(side -2) --++(side -1) --++(side 4);
\draw (corner 1) --++(side 	3);
\draw (corner 2) --++(side 3);
\draw (corner 4) --++(side -1) --++(side -2);
\draw (corner 6) --++(side -4);
\end{tikzpicture}
\caption{The eight rhombic tilings of $X(\protect\longelt{4})$.}
\label{fig:rhombic tilings of X(longelt{4})}
\end{figure}
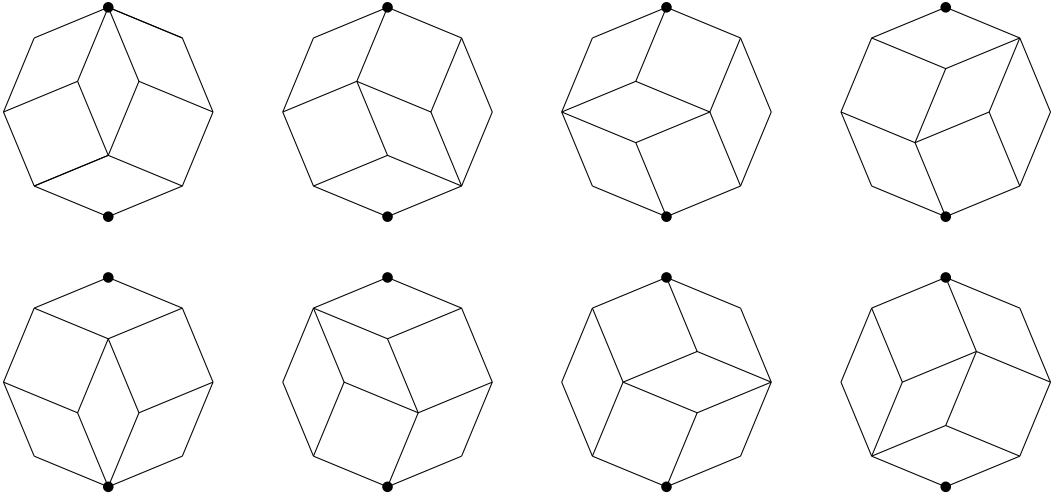

Our first step in the study of perimeter tiles for $w_0$ will be to bound (sharply) the number of perimeter tiles that may appear in elements of $T(w_0)$. By Corollary~\ref{cor:meaning of perimeter tiles}, this will have implications for commutation classes of $R(w_0)$.

In light of Lemma~\ref{lem:can rotate}, one perspective on Corollary~\ref{cor:always have left and right perimeter tiles} is that at least one perimeter tile appears in its entirety among any $n$ consecutive boundary edges of $X(\longelt{n})$. In other words, there cannot be $n-1$ consecutive boundary edges that intersect no perimeter tiles. (This statement need not be true for arbitrary permutations $w$, because it relies on the convexity of $X(\longelt{n})$.) This allows us to advance the result of Corollary~\ref{cor:always have left and right perimeter tiles} in this setting.

Note that in the following theorem, and in some subsequent results, we require $n > 2$. This is because $X(21)$ is, itself, a rhombus, and its sole rhombic tiling is just a single rhombus. That single tile is, simultaneously, a perimeter tile of all four types, and would be overcounted.

\begin{theorem}\label{thm:lower bound for total perimeter tiles}
For any $n > 2$,
$$\min\left\{\perim(T) : T \in T(\longelt{n})\right\} = 3.$$
\end{theorem}

\begin{proof}
From Corollary~\ref{cor:always have left and right perimeter tiles}, we know that $T$ has at least one left-perimeter tile and at least one right-perimeter tile. Using Lemma~\ref{lem:can rotate}, rotate the $2n$-gon so that the left-perimeter tile becomes a bottom-perimeter tile. Corollary~\ref{cor:always have left and right perimeter tiles} says that this new tiling must also have left- and right-perimeter tiles. Therefore, $\perim(T) \ge 3$.

As shown in Figure~\ref{fig:three perimeter tiles}, there are tilings for which $\perim(T) = 3$, and it is possible for only two of the required three perimeter tiles in $T$ to be side-perimeter tiles.
\end{proof}

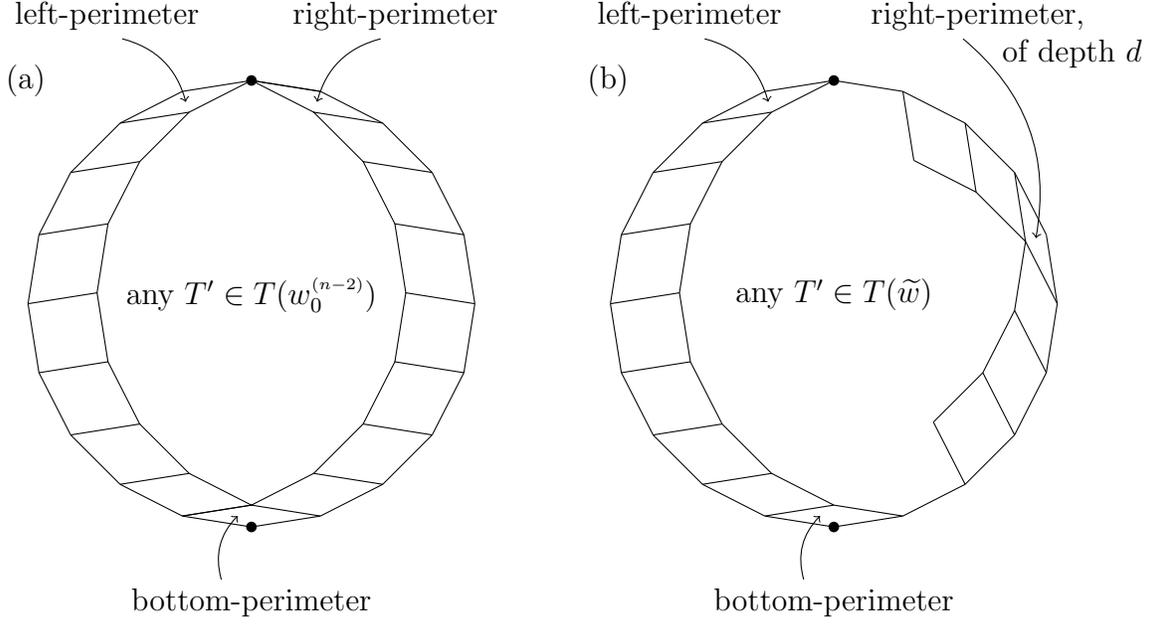
\begin{figure}[htbp]
\begin{tikzpicture}
\node[draw=none,minimum size=6cm,regular polygon,regular polygon sides=20] (a) {};
\foreach \x in {(a.side 1),(a.side 11)} {\fill \x circle (2pt);}
\foreach \x in {1,...,20} {\coordinate (corner \x) at (a.side \x);}
\foreach \y [evaluate=\y as \x using \y+1] in {1,...,10} {\coordinate (side \y) at ($(corner \x)-(corner \y)$);}
\foreach \y [evaluate=\y as \x using \y+1] in {1,...,10} {\coordinate (side -\y) at ($(corner \y)-(corner \x)$);}
\foreach \y [evaluate=\y as \x using \y+1] in {1,...,19} {\draw (corner \y) -- (corner \x);}
\draw (corner 20) -- (corner 1);
\foreach \y [evaluate=\y as \x using \y+1] in {2,...,9} {\draw (corner \x) --++(side -1) --++(side -\y);}
\foreach \y [evaluate=\y as \x using \y-11] in {12,...,20} {\draw (corner \y) --++(side -10) --++(side \x);}
\draw (corner 2)++(-1,1) node {left-perimeter};
\path[->] (corner 2)++(-.8,.7) edge[bend left] ($(corner 1)!0.5!(corner 3)$);
\draw (corner 20)++(1,1) node {right-perimeter};
\path[->] (corner 20)++(.8,.7) edge[bend right] ($(corner 19)!0.5!(corner 1)$);
\draw (corner 11)++(0,-1) node {bottom-perimeter};
\path[->] (corner 11)++(-.4,-.7) edge[bend left] ($(corner 10)!0.4!(corner 12)$);
\draw ($(corner 1)!0.5!(corner 11)$)++(0,.125) node {any $T' \in T(\longelt{n-2})$};
\draw (corner 1)++(-3,0) node {(a)};
\end{tikzpicture}
\hspace{.25in}
\begin{tikzpicture}
\node[draw=none,minimum size=6cm,regular polygon,regular polygon sides=20] (a) {};
\foreach \x in {(a.side 1),(a.side 11)} {\fill \x circle (2pt);}
\foreach \x in {1,...,20} {\coordinate (corner \x) at (a.side \x);}
\foreach \y [evaluate=\y as \x using \y+1] in {1,...,10} {\coordinate (side \y) at ($(corner \x)-(corner \y)$);}
\foreach \y [evaluate=\y as \x using \y+1] in {1,...,10} {\coordinate (side -\y) at ($(corner \y)-(corner \x)$);}
\foreach \y [evaluate=\y as \x using \y+1] in {1,...,19} {\draw (corner \y) -- (corner \x);}
\draw (corner 20) -- (corner 1);
\foreach \y in {13,14,15,16} {\draw (corner \y) --++(side -7);}
\foreach \y in {18,19,20} {\draw (corner \y) --++(side 6);}
\draw (corner 1)++(side 6)++(side 10) --++(side 9) --++(side 8) --++(side 5) --++(side 4) --++(side 3);
\draw (corner 12) --++(side -10);
\draw (corner 10) --++(side -1) coordinate (a);
\draw (a) --++(side -9) --++(side -8) --++(side -7) --++(side -6) --++(side -5) --++(side -4) --++(side -3) --++(side -2);
\foreach \y in {3,4,5,6,7,8,9} {\draw (corner \y) --++(side -1);}
\draw (corner 2)++(-1,1) node {left-perimeter};
\path[->] (corner 2)++(-.8,.7) edge[bend left] ($(corner 1)!0.5!(corner 3)$);
\draw (corner 20)++(1,1) node {right-perimeter,};
\draw (corner 20)++(2.25,.5) node {of depth $d$};
\path[->] (corner 20)++(.8,.7) edge[bend left] ($(corner 18)!0.5!(corner 16)$);
\draw (corner 11)++(0,-1) node {bottom-perimeter};
\path[->] (corner 11)++(-.4,-.7) edge[bend left] ($(corner 10)!0.4!(corner 12)$);
\draw ($(corner 1)!0.5!(corner 11)$)++(0,.125) node {any $T' \in T(\widetilde{w})$};
\draw (corner 1)++(-3,0) node {(b)};
\end{tikzpicture}
\caption{Two families of elements in $T(\protect\longelt{n})$ minimizing the number of perimeter tiles. For (b), in which the right-perimeter tile has generic depth $d$, the permutation $\widetilde{w}$ is $(n-1)(n-d-1)(n-2)(n-3)\cdots(n-d+1)(n-d-2)\cdots2(n-d) 1\in S_{n-1}$.}
\label{fig:three perimeter tiles}
\end{figure}

We now give a sharp upper bound to the total number of perimeter tiles that can appear in rhombic tilings of $X(w_0)$.

\begin{theorem}\label{thm:upper bound for total perimeter tiles}
For any $n \ge 2$,
\begin{equation}\label{eqn:max perimeter tiles}
\max\left\{\perim(T) : T \in T(\longelt{n})\right\} = 2 \left\lfloor \frac{n-1}{2}\right\rfloor + 1 =
\begin{cases}
n & \text{if $n$ is odd, and}\\
n-1 & \text{if $n$ is even,}
\end{cases}
\end{equation}
and at most $n-1$ of those perimeter tiles can be side-perimeter tiles.
\end{theorem}

\begin{proof}
Recall, from Proposition~\ref{prop:at most n perimeter rhombi}, that there are at most $n$ perimeter tiles.

Consider first when $n$ is even. For there to be exactly $n$ perimeter tiles, the entire boundary of $X(\longelt{n})$ would belong to perimeter tiles. This would mean that the edge labels of the perimeter tiles, read in counterclockwise order from the top vertex, are either
$$\{1,2\},\{3,4\},\ldots,\{n-1,n\}, \{1,2\}, \{3,4\}, \ldots, \{n-1,n\},$$
or
$$\{2,3\},\{4,5\},\ldots,\{n-2,n-1\},\{n,1\},\{2,3\},\{4,5\},\ldots,\{n-2,n-1\},\{n,1\}.$$
Each option would violate Corollary~\ref{cor:rhombi can't have same labels}, so $\perim(T) \le n-1$.

If $n$ is odd, and there are $n$ perimeter tiles, then the entire boundary of $X(\longelt{n})$ must belong to perimeter tiles. In particular, the parity of $n$ means that one of those tiles has edge labels $\{1,n\}$, and is either a top- or a bottom-perimeter tile. Therefore there can be at most $n-1$ side-perimeter tiles.

The number of perimeter tiles in $T$ is bounded by the quantity in Equation~\eqref{eqn:max perimeter tiles}, and it remains to show that this bound is achievable. We do so in Figure~\ref{fig:max perimeter tiles}.
\end{proof}

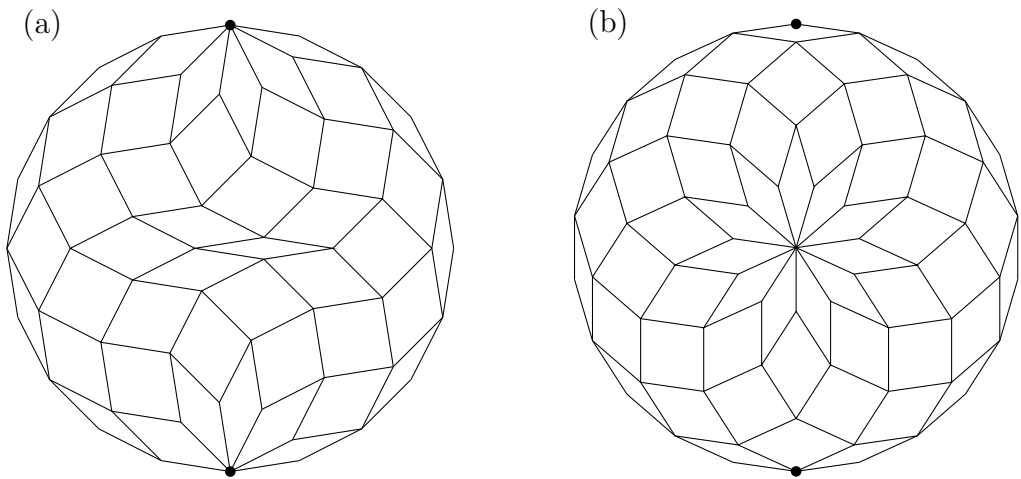
\begin{figure}[htbp]
\begin{tikzpicture}
\node[draw=none,minimum size=6cm,regular polygon,regular polygon sides=20] (a) {};
\foreach \x in {(a.side 1),(a.side 11)} {\fill \x circle (2pt);}
\foreach \x in {1,...,20} {\coordinate (corner \x) at (a.side \x);}
\foreach \y [evaluate=\y as \x using \y+1] in {1,...,10} {\coordinate (side \y) at ($(corner \x)-(corner \y)$);}
\foreach \y [evaluate=\y as \x using \y+1] in {1,...,10} {\coordinate (side -\y) at ($(corner \y)-(corner \x)$);}
\foreach \y [evaluate=\y as \x using \y+1] in {1,...,19} {\draw (corner \y) -- (corner \x);}
\draw (corner 20) -- (corner 1);
\draw (corner 2) --++(side 3) --++(side 2) --++(side 5) --++(side 4) --++(side 7) --++(side 6) --++(side 9) --++(side 8);
\draw (corner 1) --++(side 3) coordinate (corner a) --++(side 1) --++(side 5) --++(side 2) --++(side 7) --++(side 4) --++(side 9) --++(side 6) --++(side 10) --++(side 8);
\draw (corner a) --++(side 5) --++(side 1) --++(side 7) --++(side 2) --++(side 9) --++(side 4) --++(side 10) --++(side 6);
\draw (corner 1) --++(side 5) coordinate (corner b) --++(side 3) --++(side 7) --++(side 1) --++(side 9) --++(side 2) --++(side 10) --++(side 4) --++(side 8) --++(side 6);
\draw (corner b) --++(side 7) --++(side 3) --++(side 9) --++(side 1) --++(side 10) --++(side 2) --++(side 8) --++(side 4);
\draw (corner 1) --++(side 7) coordinate (corner c) --++(side 5) --++(side 9) --++(side 3) --++(side 10) --++(side 1) --++(side 8) --++(side 2) --++(side 6) --++(side 4);
\draw (corner c) --++(side 9) --++(side 5) --++(side 10) --++(side 3) --++(side 8) --++(side 1) --++(side 6) --++(side 2);
\draw (corner 1) --++(side 9) coordinate (corner d) --++(side 7) --++(side 10) --++(side 5) --++(side 8) --++(side 3) --++(side 6) --++(side 1) --++(side 4) --++(side 2);
\draw (corner d) --++(side 10) --++(side 7) --++(side 8) --++(side 5) --++(side 6) --++(side 3) --++(side 4) --++(side 1);
\draw (corner 1)++(-2.5,0) node {(a)};
\end{tikzpicture}
\hspace{.5in}
\begin{tikzpicture}
\node[draw=none,minimum size=6cm,regular polygon,regular polygon sides=22] (a) {};
\foreach \x in {(a.side 1),(a.side 12)} {\fill \x circle (2pt);}
\foreach \x in {1,...,22} {\coordinate (corner \x) at (a.side \x);}
\foreach \y [evaluate=\y as \x using \y+1] in {1,...,11} {\coordinate (side \y) at ($(corner \x)-(corner \y)$);}
\foreach \y [evaluate=\y as \x using \y+1] in {1,...,11} {\coordinate (side -\y) at ($(corner \y)-(corner \x)$);}
\foreach \y [evaluate=\y as \x using \y+1] in {1,...,21} {\draw (corner \y) -- (corner \x);}
\draw (corner 22) -- (corner 1);
\draw (corner 2) --++(side 3) --++(side 2) --++(side 5) --++(side 4) --++(side 7) --++(side 6) --++(side 9) --++(side 8) --++(side 11) --++(side 10) --++(side -2) --++(side -1) --++(side -4) --++(side -3) --++(side -6) --++(side -5) --++(side -8) --++(side -7) --++(side -10) --++(side -9) --++(side 1) coordinate (corner a) --++(side -11);
\draw (corner a) --++(side 3) --++(side -11) --++(side 5) --++(side 2) --++(side 7) --++(side 4) --++(side 9) --++(side 6) --++(side 11) --++(side 8) --++(side -2) --++(side 10) --++(side -4) --++(side -1) --++(side -6) --++(side -3) --++(side -8) --++(side -5) --++(side -10) --++(side -7) --++(side 1) coordinate (corner b) --++(side -9);
\draw (corner b) --++(side 3) --++(side -9) --++(side 5) --++(side -11) --++(side 7) --++(side 2) --++(side 9) --++(side 4) --++(side 11) --++(side 6) --++(side -2) --++(side 8) --++(side -4) --++(side 10) --++(side -6) --++(side -1) --++(side -8) --++(side -3) --++(side -10) --++(side -5) --++(side 1) coordinate (corner c) --++(side -7);
\draw (corner c) --++(side 3) --++(side -7) --++(side 5) --++(side -9) --++(side 7) --++(side -11) --++(side 9) --++(side 2) --++(side 11) --++(side 4) --++(side -2) --++(side 6) --++(side -4) --++(side 8) --++(side -6) --++(side 10) --++(side -8) --++(side -1) --++(side -10) --++(side -3) --++(side 1) coordinate (corner d) --++(side -5);
\draw (corner d) --++(side 3) coordinate (corner e) --++(side -5);
\draw (corner e) --++(side -7);
\draw (corner e) --++(side -9);
\draw (corner e) --++(side -11);
\draw (corner e) --++(side 2);
\draw (corner e) --++(side 4);
\draw (corner e) --++(side 6);
\draw (corner e) --++(side 8);
\draw (corner e) --++(side 10);
\draw (corner e) --++(side -1);
\draw (corner 1)++(-2.5,0) node {(b)};
\end{tikzpicture}
\caption{Rhombic tilings of $X(\protect\longelt{n})$ maximizing the number of perimeter tiles. The figures are drawn for $n = 10$ (tiling (a)) and $n = 11$ (tiling (b)), and can easily be generalized. As shown in Theorem~\ref{thm:upper bound for total perimeter tiles}, this maximum is $n-1 = 9$ for the tiling in (a), and $n = 11$ for the tiling in (b).}
\label{fig:max perimeter tiles}
\end{figure}

In light of these bounds, we can make the following conclusions.

\begin{corollary}\label{cor:max/min implications for commutation classes}
For any $n \ge 2$, there exist commutation classes $C,C' \in C(\longelt{n})$ with the following properties.
\begin{enumerate}\renewcommand{\labelenumi}{(\alph{enumi})}
\item All elements of $C$ have the same leftmost letter, and all have the same rightmost letter, and exactly one copy of either $s_1$ or $s_{n-1}$ appears; and 
\item If $n$ is even, then the set of leftmost letters that appear in elements of $C'$ is $\mathcal{L}$, and the set of rightmost letters that appear in elements of $C'$ is $\mathcal{R}$, where $\{\mathcal{L},\mathcal{R}\} =\{\mathcal{O},\mathcal{E}\}$ for
\begin{align*}
\mathcal{O} &= \{s_1,s_3,s_5,\ldots\} \cap \{s_1,s_2,\ldots, s_{n-1}\}\text{ and}\\
\mathcal{E} &= \{s_2,s_4,s_6,\ldots\} \cap \{s_1,s_2,\ldots, s_{n-1}\}.
\end{align*}
If $n$ is odd, then the set of leftmost letters that appear in elements of $C'$ and the set of rightmost letters that appear in elements of $C'$ are either both $\mathcal{O}$ or both $\mathcal{E}$. If they are $\mathcal{O}$ (respectively, $\mathcal{E}$), then elements of $C'$ contain exactly one copy of $s_{n-1}$ (resp., $s_1$).
\end{enumerate}
\end{corollary}

\begin{proof}
\begin{enumerate}\renewcommand{\labelenumi}{(\alph{enumi})}
\item This follows from Corollary~\ref{cor:meaning of perimeter tiles} and Theorem~\ref{thm:lower bound for total perimeter tiles}.
\item This follows from Corollary~\ref{cor:meaning of perimeter tiles} and Theorem~\ref{thm:upper bound for total perimeter tiles}. 
\end{enumerate}
\end{proof}

\begin{example}
The following commutation classes in $C(4321)$ illustrate Corollary~\ref{cor:max/min implications for commutation classes}.
\begin{enumerate}\renewcommand{\labelenumi}{(\alph{enumi})}
\item $\{s_1s_2s_3s_1s_2s_1, s_1s_2s_1s_3s_2s_1\}$
\item $\{s_2s_3s_1s_2s_1s_3, s_2s_1s_3s_2s_1s_3, s_2s_3s_1s_2s_3s_1, s_2s_1s_3s_2s_3s_1\}$
\end{enumerate} 
\end{example}

We close this section with a brief look at the average number of perimeter tiles that appear in a rhombic tiling of $X(w_0)$. Consider the set $T(w_0)$ with the uniform probability distribution. Viewing $\perim(T)$ as a random variable on rhombic tilings $T \in T(w_0)$ that counts the perimeter tiles, answering that question amounts to understanding
$$\mathbb{E}[\perim(T)].$$
We could similarly define $\mathbb{E}[\leftperim(T)]$, $\mathbb{E}[\rightperim(T)]$, $\mathbb{E}[\topperim(T)]$, and $\mathbb{E}[\bottomperim(T)]$. This data addresses questions like: what is the average number of letters that can appear as leftmost letters in reduced words within a single commutation class of $R(w_0)$? What is the probability that the reduced decompositions in a commutation class of $R(w_0)$ contain exactly one copy of $s_1$? In order to discuss these questions, we borrow terminology from \cite{elnitsky}.

\begin{definition}
Let $w \in \mathfrak{S}_n$ be a permutation and $T \in T(w)$ a rhombic tiling of $X(w)$. An $n$-edge path in $T$ from the top vertex of $X(w)$ to the bottom vertex is a \emph{border}.
\end{definition}

In \cite{elnitsky}, Elnitsky recognizes that rhombic tilings of convex $2n$-gons can be generated recursively, using borders. Similar recursion can be used to compute the expected values listed above, yielding double sums of border enumerations, taken over subtilings of $T(\longelt{n-2})$. We omit these formulations here, as their utility may not be worth the page-length, and instead close with data related to these issues, for small values of $n$.

For $n \in [2,6]$, the number of perimeter tiles of each (fixed depth) type among all $T(\longelt{n})$ is $1$, $1$, $3$, $20$, and $268$. Aggregate perimeter tile data, as well as $|T(\longelt{n})|$ itself, appears in Table~\ref{table:tiling frequency and number of tilings} for the same range of $n$. 
\begin{table}[htbp]
$\begin{array}{c||c|c|c|c}
&
& \sum \topperim(T) \hspace{.25in}\ 
& \sum \leftperim(T) \hspace{.25in}\ 
&\\
\raisebox{0in}[.2in][.1in]{} n & \sum \perim(T) 
& \ \ = \sum \bottomperim(T)
& \ \ = \sum \rightperim(T)
& |T(\longelt{n})|\\
\hline
\hline
\raisebox{0in}[.2in][.1in]{} 2 & 1 & 1 & 1 & 1\\
\hline
\raisebox{0in}[.2in][.1in]{} 3 & 6 & 1 & 2 & 2\\
\hline
\raisebox{0in}[.2in][.1in]{} 4 & 24 & 3 & 9 & 8\\
\hline
\raisebox{0in}[.2in][.1in]{} 5 & 200 & 20 & 80 & 62\\
\hline
\raisebox{0in}[.2in][.1in]{} 6 & 3216 & 268 & 1340 & 908
\end{array}$
\caption{The total number of perimeter tiles, overall and by type, among all $T(\protect\longelt{n})$, and $|T(\protect\longelt{n})|$ itself, for $2 \le n \le 6$. All sums are taken over $T \in T(w)$.}
\label{table:tiling frequency and number of tilings}
\end{table}
The columns of Table~\ref{table:tiling frequency and number of tilings} are entries A320944, A320945, A320946, and A006245, respectively, in \cite{oeis}. The last column of the table, which enumerates objects whose relevance precedes that of our present work, was studied by others including Knuth~\cite[\S 9]{knuth} and Ziegler~\cite[Table 2]{ziegler}. Outside of special cases and bounds like $(2^{n^2/6-O(n)},2^{n^2+n})$ (see \cite{knuth}), these values are notoriously difficult to calculate. For example, only the values for $n \le 15$ appear in \cite{oeis}.

\section{Isoperimetric results for rhombic tilings}\label{sec:general perm}

In the context of the rhombic tilings discussed in this paper, recall that there are natural interpretations of ``area'' and ``perimeter,'' as discussed in Definitions~\ref{defn:tiling-based area} and~\ref{defn:perimeter counting functions}. Because the number of rhombi in any $T \in T(w)$ is equal to the length $\ell(w)$ of the permutation, the area depends only on $w$, and is independent of the choice of $T \in T(w)$.

This sets up a framework for analyzing tiling-based isoperimetric properties. One might ask, for a given permutation $w$, what are the minimal (and, for that matter, possible) perimeters obtainable among $T \in T(w)$, and which $T$ achieve them. One might also fix an area $A$ and consider the $w$ and $T$ achieving
$$\min\{\perim(T) :  \ell(w) = A \text{ and } T \in T(w)\}.$$

The following example of this sort of analysis will be called on later in this work.

\begin{example}\label{ex:skinny polygons, 2 perimeter tiles}
For any $A > 1$, there are permutations of area $A$ and strong perimeter $2$ (that is, every rhombic tiling of Elnitsky's polygon has exactly two perimeter tiles):
$$23\cdots A(A+1)1 \hspace{.25in} \text{and} \hspace{.25in} (A+1)12\cdots A.$$
Indeed, for each of these permutations, the Elnitsky polygon has a single rhombic tiling, as depicted in Figure~\ref{fig:skinny polygons, 2 perimeter tiles}.
\end{example}

\begin{figure}[htbp]
\begin{tikzpicture}
\begin{scope}
\clip (.2,2.1) rectangle (-2,-2.1);
\node[draw=none,minimum size=4cm,regular polygon,regular polygon sides=14] (a) {};
\foreach \x in {(a.side 1),(a.side 8)} {\fill \x circle (2pt);}
\end{scope}
\foreach \x in {1,2,3,4,5,6,7,8} {\coordinate (corner \x) at (a.side \x);}
\foreach \y [evaluate=\y as \x using \y+1] in {1,2,3,4,5,6,7} {\coordinate (side \y) at ($(corner \x)-(corner \y)$);}
\foreach \y [evaluate=\y as \x using \y+1] in {1,2,3,4,5,6,7} {\coordinate (side -\y) at ($(corner \y)-(corner \x)$);}
\fill[black!20] (corner 1) -- (corner 2)  -- (corner 3) --++(side -1) --++(side -2); 
\fill[black!20] (corner 8) -- (corner 7) --++(side -1) --++(side 7) --++(side 1);
\draw (corner 1) -- (corner 2) -- (corner 3) -- (corner 4) -- (corner 5) -- (corner 6) -- (corner 7) -- (corner 8);
\foreach \x in {3,...,7} {\draw (corner \x) --++(side -1);}
\draw (corner 8) --++(side -1) --++(side -7) --++(side -6) --++(side -5) --++(side -4) --++(side -3) --++(side -2);
\end{tikzpicture}
\hspace{1in}
\begin{tikzpicture}
\begin{scope}
\clip (.2,2.1) rectangle (-2,-2.1);
\node[draw=none,minimum size=4cm,regular polygon,regular polygon sides=14] (a) {};
\foreach \x in {(a.side 1),(a.side 8)} {\fill \x circle (2pt);}
\end{scope}
\foreach \x in {1,2,3,4,5,6,7,8} {\coordinate (corner \x) at (a.side \x);}
\foreach \y [evaluate=\y as \x using \y+1] in {1,2,3,4,5,6,7} {\coordinate (side \y) at ($(corner \x)-(corner \y)$);}
\foreach \y [evaluate=\y as \x using \y+1] in {1,2,3,4,5,6,7} {\coordinate (side -\y) at ($(corner \y)-(corner \x)$);}
\fill[black!20] (corner 1) -- (corner 2) --++(side 7) --++(side -1) --++(side -7); 
\fill[black!20] (corner 8) -- (corner 7) -- (corner 6) --++(side 7) --++(side 6); 
\draw (corner 1) -- (corner 2) -- (corner 3) -- (corner 4) -- (corner 5) -- (corner 6) -- (corner 7) -- (corner 8);
\foreach \x in {2,...,6} {\draw (corner \x) --++(side 7);}
\draw (corner 1) --++(side 7) --++(side 1) --++(side 2) --++(side 3) --++(side 4) --++(side 5) --++(side 6);
\end{tikzpicture}
\caption{The only rhombic tilings of $X(2345671)$ and $X(7123456)$, each of which has exactly two perimeter tiles. The perimeter tiles have been shaded.}
\label{fig:skinny polygons, 2 perimeter tiles}
\end{figure}
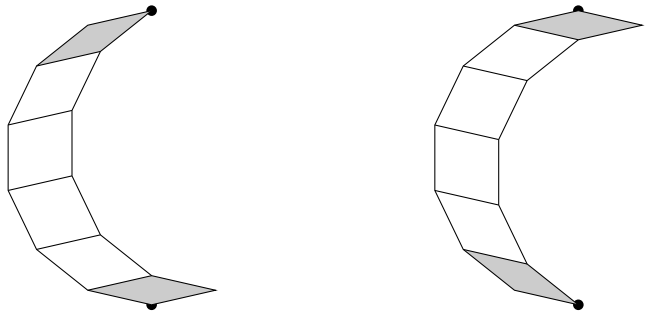

As discussed after Corollary~\ref{cor:always have left and right perimeter tiles}, if the area of a permutation is greater than $1$, then the strong perimeter of any of its tilings is greater than $1$ as well. Thus, apart from trivial cases, the strong perimeter of any tiling is always at least $2$. In other words, the permutations in Example~\ref{ex:skinny polygons, 2 perimeter tiles} illustrate minimal strong perimeters.

Combining Theorem~\ref{thm:lower bound for total perimeter tiles} with Example~\ref{ex:skinny polygons, 2 perimeter tiles} illustrates an important contrast between the contour-based isoperimetric problem and the analogous problem in this discrete setting. 

\begin{corollary}
Among all fixed-area tilings of Elnitsky polygons, the tilings that minimize strong perimeter are not the tilings of $X(w_0)$.
\end{corollary}

Example~\ref{ex:skinny polygons, 2 perimeter tiles} gives two classes of permutations whose tilings achieve minimal strong perimeter. Moreover, for each permutation described in that lemma, the \emph{only} rhombic tiling of Elnitsky's polygon has exactly two (strong) perimeter tiles. This raises the question: what other permutations have this property? In light of the significance of perimeter tiles discussed in Corollary~\ref{cor:meaning of perimeter tiles}, there are two variants of this question that are particularly interesting.

\begin{question}\label{question:min perimeter tiles}\
\begin{enumerate}\renewcommand{\labelenumi}{(\alph{enumi})}
\item What are the permutations $w$ for which all rhombic tilings of $X(w)$ have exactly two perimeter tiles?
\item What are the permutations $w$ for which all rhombic tilings of $X(w)$ have exactly two side-perimeter tiles?
\end{enumerate}
\end{question}

Example~\ref{ex:skinny polygons, 2 perimeter tiles} addresses both parts of Question~\ref{question:min perimeter tiles}, and we now characterize all permutations that satisfy either question. We begin by recalling Corollary~\ref{cor:meaning of perimeter tiles}(b), which described how right-perimeter tiles relate to convex portions of the rightside boundary of $X(w)$. We can also find top- and bottom-perimeter tiles when Requirement~\ref{req:contiguity}(b) is satisfied.

\begin{lemma}\label{lem:top/bottom perim possible}
Let $w \in \mathfrak{S}_n$. If $w(1) \neq 1$, then there exists a rhombic tiling of $X(w)$ with a top-perimeter tile. If $w(n) \neq n$, then there exists a rhombic tiling of $X(w)$ with a bottom-perimeter tile.
\end{lemma}

\begin{proof}
We prove this by construction. Suppose $w(1) \neq 1$. Define $\widetilde{w} \in \mathfrak{S}_{n-1}$ as
$$\widetilde{w}(i) := \begin{cases}
w(i+1) & w(i+1) < w(1), \text{ and}\\
w(i+1)-1 & w(i+1) > w(1).
\end{cases}$$
Build a tiling $T \in T(w)$ using tiles with edge labels $\{i,w(1)\}$ along the leftside boundary of $X(w)$, for $i < w(1)$, and tiling the rest of $X(w)$ as in any rhombic tiling of $X(\widetilde{w})$. As constructed, this $T$ contains a top-perimeter tile. An example of this procedure in which $w \in \mathfrak{S}_7$ and $w(1) = 5$ appears in Figure~\ref{fig:top tile possible}.

The existence of bottom-perimeter tiles when $w(n) \neq n$ can be proved analogously.
\end{proof}

\begin{figure}[htbp]
\begin{tikzpicture}
\draw[white] (-3,0) node {any $T' \in T(\widetilde{w})$};
\node[draw=none,minimum size=4cm,regular polygon,regular polygon sides=14] (a) {};
\foreach \x in {(a.side 1),(a.side 8)} {\fill \x circle (2pt);}
\foreach \x in {1,...,14} {\coordinate (corner \x) at (a.side \x);}
\foreach \y [evaluate=\y as \x using \y+1] in {1,...,7} {\coordinate (side \y) at ($(corner \x)-(corner \y)$);}
\foreach \y [evaluate=\y as \x using \y+1] in {1,...,7} {\coordinate (side -\y) at ($(corner \y)-(corner \x)$);}
\foreach \y [evaluate=\y as \x using \y+1] in {1,...,7} {\draw (corner \y) -- (corner \x);}
\draw (corner 1) --++(side 5) coordinate (a) --++(side 1) --++(side 2) --++(side 3) --++(side 4);
\foreach \x in {2,3,4} {\draw (corner \x) --++(side 5);}
\draw[decorate,decoration={snake,amplitude=1mm,segment length=4mm,pre length=1mm,post length=1mm}] (a) to[bend left] (corner 8);
\draw (3,0) node {any $T' \in T(\widetilde{w})$};
\path[->] (3,.4) edge[bend right] (0,-.5);
\end{tikzpicture}
\caption{Constructing $T \in T(w)$ with a top-perimeter tile when $w(1) \neq 1$, as described in the proof of Lemma~\ref{lem:top/bottom perim possible}.}
\label{fig:top tile possible}
\end{figure}
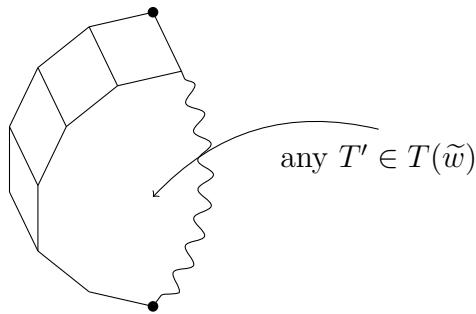

We now turn to Question~\ref{question:min perimeter tiles}(a), whose answer will bring Example~\ref{ex:skinny polygons, 2 perimeter tiles} to mind.

\begin{theorem}\label{thm:exactly two perim tiles}
For any $n > 2$, a permutation $w \in \mathfrak{S}_n$ has strong perimeter $2$ (that is, is such that all rhombic tilings of $X(w)$ have exactly two perimeter tiles) if and only if $w$ is
$$n123\cdots (n-1) \hspace{.25in} \text{or} \hspace{.25in} 234\cdots n1.$$
\end{theorem}

\begin{proof}
From Example~\ref{ex:skinny polygons, 2 perimeter tiles}, the permutations $n123\cdots (n-1)$ and $234\cdots n1$ have the desired perimeter property.

Now suppose that all rhombic tilings of $X(w)$ have exactly two perimeter tiles. By Requirement~\ref{req:contiguity}(b) and Lemma~\ref{lem:top/bottom perim possible}, there exists $T \in T(w)$ with a top-perimeter tile. Therefore, by Corollary~\ref{cor:always have left and right perimeter tiles}, if $T$ is to have exactly two perimeter tiles, then this top-perimeter tile must also be a left- or right-perimeter tile. To be a left-perimeter tile (call this Case 1), we would have to have $w(1) = 2$, whereas to be a right-perimeter tile (Case 2), we would have to have $w(2) = 1$. In each case, that top-perimeter tile is consequently forced in any rhombic tiling of $X(w)$.

By a similar argument, we find that any rhombic tiling of $X(w)$ contains a bottom-perimeter tile. Recall, again, Corollary~\ref{cor:always have left and right perimeter tiles}. In Case 1, then, this bottom-perimeter tile must also be a right-perimeter tile and so $w(n-1) = n$. On the other hand, in Case 2, the bottom-perimeter tile would also be a left-perimeter tile, and so $w(n) = n-1$.

In each of these cases, we have identified two (distinct, because $n > 2$) perimeter tiles that exist in every element of $T(w)$. 

To satisfy the requirements of our hypothesis about the strong perimeter of $w$, Corollary~\ref{cor:meaning of perimeter tiles}(b) implies that there can be no $w(d) > w(d+1)$ besides $d = n-1$ in Case 1, and $d = 1$ in Case 2. Thus $w = 234\cdots n1$ in Case 1, and $w = n123\cdots (n-1)$ in Case 2.
\end{proof}

We now see that Question~\ref{question:min perimeter tiles}(a) was highly restrictive --- there only two permutations in $\mathfrak{S}_n$ with the desired property, and the Elnitsky polygon for each of those permutations only has one rhombic tiling.

In contrast to the result of Theorem~\ref{thm:exactly two perim tiles}, Question~\ref{question:min perimeter tiles}(b) seeks to address a more Coxeter-focused version of the isoperimetric problem in the context of these rhombic tilings. More specifically, this question seeks to know when there is only one choice of leftmost letter and only one choice of rightmost letter within any commutation class of reduced decompositions of $w \in \mathfrak{S}_n$, without further restrictions on the number of $s_1$ or $s_{n-1}$ letters that might appear in those reduced decompositions.

\begin{theorem} \label{thm:exactly two side-perim tiles}
For any $n > 2$, a permutation $w \in \mathfrak{S}_n$ is such that all rhombic tilings of $X(w)$ have exactly two side-perimeter tiles if and only if $w$ has the form 
$$(k+1)(k+2)\cdots n 12 \cdots k \hspace{.25in} \text{or} \hspace{.25in} (k+1)(k+2)\cdots n k 12\cdots (k-1)$$
for some $k \in [n-1]$. 
\end{theorem} 

\begin{proof}
First note that by Corollary~\ref{cor:meaning of perimeter tiles}, any permutation of the type listed in the theorem statement does indeed have the desired property:
\begin{itemize}
\item $T \in T((k+1)(k+2)\cdots n 12 \cdots k)$ can (in fact, must) only have a right-perimeter tile of depth $n-k$, and no other. Similarly, it can (in fact, must) only have a left-perimeter tile of depth $k$. The edge labels of the right-perimeter tile are $\{1,n\}$, and the edge labels of that left-perimeter tile are $\{k, k+1\}$.
\item $T \in T((k+1)(k+2)\cdots n k 12\cdots (k-1))$ can only have right-perimeter tiles of depths $n-k$ or $n-k+1$. Because these depths differ by $1$, they cannot both be right-perimeter tiles in $T$. Similarly, $T$ can only have left-perimeter tiles of depths $k-1$ or $k$. Because these depths differ by $1$, they cannot both be left-perimeter tiles in $T$. The edge labels of the right-perimeter tile are $\{k,n\}$ or $\{1,k\}$, respectively, and the edge labels of the left-perimeter tile are $\{k-1,k\}$ or $\{k,k+1\}$, respectively.
\end{itemize}

Suppose that all rhombic tilings of $X(w)$ have exactly two side-perimeter tiles, and recall that $w(1) \neq 1$ and $w(n) \neq n$ by Requirement~\ref{req:contiguity}(b). By Corollary~\ref{cor:always have left and right perimeter tiles}, one of these side-perimeter tiles must be a left-perimeter tile, one a right-perimeter tile, and since $n > 2$ these tiles must be distinct. Suppose that $w(i) > w(i+1)$ and $w(j) > w(j+1)$ for $|i-j| > 1$. Corollary~\ref{cor:meaning of perimeter tiles}(b) means there is $T_1 \in T(w)$ with a right-perimeter tile $t_1$ at depth $i$. Because $i$ and $j$ differ by at least $2$, that corollary also produces $T_2 \in T(ws_i)$ with a right-perimeter tile $t_2$ at depth $j$. Again because $|i-j| > 1$, appending $t_1$ and $t_2$ as right-perimeter tiles to any $T \in T(ws_is_j)$ will produce a tiling in $T(w)$ with right-perimeter tiles of depths $i$ and $j$. Thus, to have exactly one right-perimeter tile in any tiling, there can be no such $\{i,j\}$. So $w$ has at most two descents, and if there are two descents then they must be consecutive.

Suppose that $w$ has exactly one descent. Then, because $w(1) \neq 1$ and $w(n) \neq n$, the permutation $w$ must be such that $w(1) < w(2) < \cdots < w(m) = n$ and $w(m+1) = 1 < w(m+2) < \cdots < w(n)$ for some $m \in [n-1]$. In particular, in the one-line notation for $w$, $w(1)$ appears to the left of $w(1) - 1$, and $w(n)+1$ appears to the left of $w(n)$. Suppose, for the moment, that $w(n) + 1 \neq w(1)$. Then $w(n) + 1 = w(i)$ for $i \in [2,m]$ and $w(1) - 1 = w(j)$ for $j \in [m+1,n-1]$. Thus $w(n) > w(1) - 1$. In fact, because $w(n) \neq w(1)$, it must be that $w(n) > w(1)$. Therefore $s_{w(n)}$ and $s_{w(1)-1}$ commute and $w$ has reduced decompositions
$$s_{w(1)-1}s_{w(n)}s_{i_3}\cdots s_{i_{\ell(w)}} \text{ and } s_{w(n)}s_{w(1)-1}s_{i_3}\cdots s_{i_{\ell(w)}},$$
differing only by a commutation. Therefore the unlabeled tiling of $X(w)$ that corresponds to both of these reduced decompositions will have two left-perimeter tiles: at depths $w(1)-1$ and $w(n)$. This is a contradiction, and so it must be that $w(n) + 1 = w(1)$, and hence
$$w = (k+1)(k+2)\cdots n 12 \cdots k.$$

Now suppose that $w$ has exactly two descents. Then, again, because $w(1) \neq 1$ and $w(n) \neq n$, the permutation $w$ must must the form $w(1) < w(2) < \cdots < w(m) = n$, $w(m+1) = k$, and $w(m+2) = 1 < w(m+3) < \cdots < w(n)$ for some $m \in [n-2]$ and $k \in [2,n-1]$. To avoid two left-perimeter tiles in the same $T \in T(w)$, Corollary~\ref{cor:meaning of perimeter tiles} tells us that these same sorts of conclusions must hold for the inverse permutation $w^{-1}$. This means that $w^{-1}(k+1) = 1$ and $w^{-1}(k-1) = n$, from which we obtain
$$w = (k+1)(k+2)\cdots n k 12\cdots (k-1),$$
as desired.
\end{proof}

Although it was not requested in Questions~\ref{question:min perimeter tiles}(a) or~(b), the results of Theorems~\ref{thm:exactly two perim tiles} and~\ref{thm:exactly two side-perim tiles} help us understand the elements of $T(w)$ in those settings. By \cite[Theorem 6.4]{tenner rdpp}, no tiling of Elnitsky's polygon for a $321$-avoiding permutation will have any subhexagons, which means that the polygon has exactly one rhombic tiling (recovering a result of \cite{bjs}). This accounts for all of the permutations from Theorems~\ref{thm:exactly two perim tiles} and~\ref{thm:exactly two side-perim tiles} except for the second category of the latter theorem, with $k > 1$. Permutations in that class have $(n-k)(k-1)$ $321$-patterns, and so there are multiple rhombic tilings of their Elnitsky polygons. The choice of left-perimeter tile (edge labels $\{k-1,k\}$ or $\{k,k+1\}$) and the choice of right-perimeter tile (edge labels $\{1,k\}$ or $\{k,n\}$) are independent of each other, as long as $2 < k < n-1$. In each case, there will be no other side-perimeter tiles, although there will be multiple ways to tile the interior of the polygon. Example polygons for the permutations described in Theorem~\ref{thm:exactly two side-perim tiles}, and what is forced in their tilings, are depicted in Figure~\ref{fig:exactly two side-perimeter tiles}.

\begin{figure}[htbp]
\begin{tikzpicture}
\begin{scope}
\clip (.2,3) rectangle (-2,-3);
\node[draw=none,minimum size=5cm,regular polygon,regular polygon sides=22] (a) {};
\foreach \x in {(a.side 1),(a.side 12)} {\fill \x circle (2pt);}
\end{scope}
\foreach \x in {1,...,12} {\coordinate (corner \x) at (a.side \x);}
\foreach \y [evaluate=\y as \x using \y+1] in {1,...,11} {\coordinate (side \y) at ($(corner \x)-(corner \y)$);}
\foreach \y [evaluate=\y as \x using \y+1] in {1,...,11} {\coordinate (side -\y) at ($(corner \y)-(corner \x)$);}
\fill[black!20] (corner 4) -- (corner 5) -- (corner 6) --++(side -4) --++(side -5); 
\fill[black!20] (corner 12)++(side -4)++(side -3)++(side -2) --++(side -1) --++(side -11) --++(side 1) --++(side 11); 
\draw (corner 1) -- (corner 2) -- (corner 3) -- (corner 4) -- (corner 5) -- (corner 6) -- (corner 7) -- (corner 8) -- (corner 9) -- (corner 10) -- (corner 11) -- (corner 12);
\draw (corner 1) --++(side 5) coordinate coordinate (corner 22) --++(side 6) coordinate (corner 21) --++(side 7) coordinate (corner 20) --++(side 8) coordinate (corner 19) --++(side 9) coordinate (corner 18) --++(side 10) coordinate (corner 17) --++(side 11) coordinate (corner 16) --++(side 1) coordinate (corner 15) --++(side 2) coordinate (corner 14) --++(side 3) coordinate (corner 13) --++(side 4);
\foreach \x in {6,...,11} {\draw (corner \x) --++(side -4) --++(side -3) --++(side -2) --++(side -1);}
\foreach \x in {2,3,4} {\draw (corner \x) --++(side 5) --++(side 6) --++(side 7) --++(side 8) --++(side 9) --++(side 10) --++(side 11);}
\draw (-2.75,2.4) node {(a)};
\draw ($(corner 1)!0.5!(corner 2)$) node[above] {$1$};
\draw ($(corner 2)!0.5!(corner 3)$) node[above] {$2$};
\draw ($(corner 3)!0.65!(corner 4)$) node[above] {$3$};
\draw ($(corner 4)!0.35!(corner 5)$) node[left] {$4$};
\draw ($(corner 5)!.5!(corner 6)$) node[left] {$5$};
\draw ($(corner 6)!.5!(corner 7)$) node[left] {$6$};
\draw ($(corner 7)!.5!(corner 8)$) node[left] {$7$};
\draw ($(corner 8)!.65!(corner 9)$) node[left] {$8$};
\draw ($(corner 9)!.35!(corner 10)$) node[below] {$9$};
\draw ($(corner 10)!0.5!(corner 11)$) node[below] {$10$};
\draw ($(corner 11)!0.5!(corner 12)$) node[below] {$11$};
\draw ($(corner 1)!.5!(corner 22)$) node[right] {$5$};
\draw ($(corner 22)!.5!(corner 21)$) node[right] {$6$};
\draw ($(corner 21)!.5!(corner 20)$) node[right] {$7$};
\draw ($(corner 20)!.35!(corner 19)$) node[right] {$8$};
\draw ($(corner 19)!.65!(corner 18)$) node[above] {$9$};
\draw ($(corner 18)!0.5!(corner 17)$) node[above] {$10$};
\draw ($(corner 17)!0.5!(corner 16)$) node[above] {$11$};
\draw ($(corner 16)!0.5!(corner 15)$) node[below] {$1$};
\draw ($(corner 15)!0.5!(corner 14)$) node[below] {$2$};
\draw ($(corner 14)!0.35!(corner 13)$) node[below] {$3$};
\draw ($(corner 13)!0.65!(corner 12)$) node[right] {$4$};

\end{tikzpicture}
\hspace{.5in}
\begin{tikzpicture}
\begin{scope}
\clip (.2,3) rectangle (-2,-3);
\node[draw=none,minimum size=5cm,regular polygon,regular polygon sides=22] (a) {};
\foreach \x in {(a.side 1),(a.side 12)} {\fill \x circle (2pt);}
\end{scope}
\foreach \x in {1,...,12} {\coordinate (corner \x) at (a.side \x);}
\foreach \y [evaluate=\y as \x using \y+1] in {1,...,11} {\coordinate (side \y) at ($(corner \x)-(corner \y)$);}
\foreach \y [evaluate=\y as \x using \y+1] in {1,...,11} {\coordinate (side -\y) at ($(corner \y)-(corner \x)$);}
\draw (corner 1) -- (corner 2) -- (corner 3) -- (corner 4) -- (corner 5) -- (corner 6) -- (corner 7) -- (corner 8) -- (corner 9) -- (corner 10) -- (corner 11) -- (corner 12);
\draw (corner 1) --++(side 5) coordinate (corner 22) --++(side 6) coordinate (corner 21) --++(side 7) coordinate (corner 20) --++(side 8) coordinate (corner 19) --++(side 9) coordinate (corner 18) --++(side 10) coordinate (corner 17) --++(side 11) coordinate (16) --++(side 4) coordinate (corner 15) --++(side 1) coordinate (corner 14)  --++(side 2) coordinate (corner 13) --++(side 3);
\draw[densely dotted] (corner 17) --++(side 4) --++(side 11);
\draw[dashed] (corner 16) --++(side 1) --++(side 4);
\draw[densely dotted] (corner 3) --++(side 4) --++(side 3);
\draw[dashed] (corner 4) --++(side 5) --++(side 4);
\draw (-2.75,2.4) node {(b)};
\draw ($(corner 1)!0.5!(corner 2)$) node[above] {$1$};
\draw ($(corner 2)!0.5!(corner 3)$) node[above] {$2$};
\draw ($(corner 3)!0.65!(corner 4)$) node[above] {$3$};
\draw ($(corner 4)!0.35!(corner 5)$) node[left] {$4$};
\draw ($(corner 5)!.5!(corner 6)$) node[left] {$5$};
\draw ($(corner 6)!.5!(corner 7)$) node[left] {$6$};
\draw ($(corner 7)!.5!(corner 8)$) node[left] {$7$};
\draw ($(corner 8)!.65!(corner 9)$) node[left] {$8$};
\draw ($(corner 9)!.35!(corner 10)$) node[below] {$9$};
\draw ($(corner 10)!0.5!(corner 11)$) node[below] {$10$};
\draw ($(corner 11)!0.5!(corner 12)$) node[below] {$11$};
\draw ($(corner 1)!.5!(corner 22)$) node[right] {$5$};
\draw ($(corner 22)!.5!(corner 21)$) node[right] {$6$};
\draw ($(corner 21)!.5!(corner 20)$) node[right] {$7$};
\draw ($(corner 20)!.35!(corner 19)$) node[right] {$8$};
\draw ($(corner 19)!.65!(corner 18)$) node[above] {$9$};
\draw ($(corner 18)!0.5!(corner 17)$) node[above] {$10$};
\draw ($(corner 17)!0.5!(corner 16)$) node[above] {$11$};
\draw ($(corner 16)!0.65!(corner 15)$) node[right] {$4$};
\draw ($(corner 15)!0.5!(corner 14)$) node[below] {$1$};
\draw ($(corner 14)!0.5!(corner 13)$) node[below] {$2$};
\draw ($(corner 13)!0.35!(corner 12)$) node[below] {$3$};
\end{tikzpicture}
\caption{Examples of $X(w)$ for permutations $w$ satisfying the statement of Theorem~\ref{thm:exactly two side-perim tiles}. Both permutations have $n = 11$ and $k = 4$, with (a) having the first form in the theorem statement and (b) having the second. The polygon in (a) has exactly one rhombic tiling because the permutation has no $321$-pattern, and its two side-perimeter tiles have been shaded. The polygon in (b) has many rhombic tilings because the permutation has twenty-one $321$-patterns, but there can only ever be one left-perimeter tile and one right-perimeter tile in any particular rhombic tiling. The possible left- and right-perimeter tiles (which can be chosen independently of each other in this case) are indicated with dashed and dotted edges in the figure.}
\label{fig:exactly two side-perimeter tiles}
\end{figure}
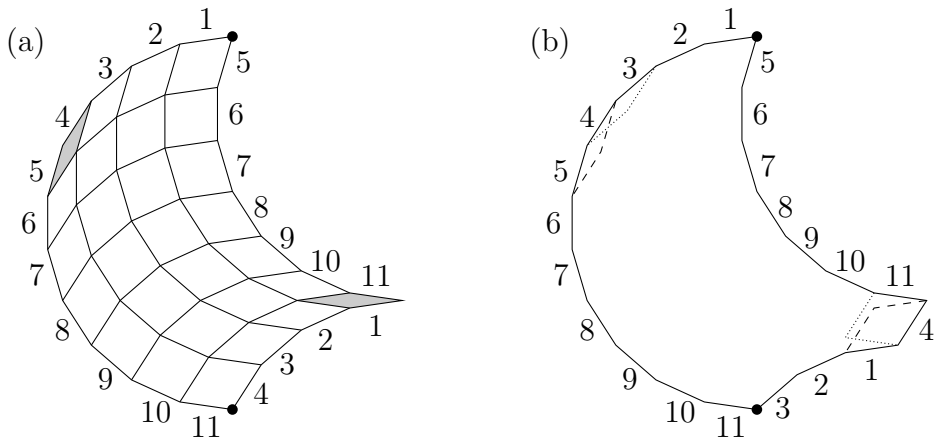

\section{Perimeter tiles in domino tilings: an introduction}\label{sec:domino}

A highly active area of mathematical research (in combinatorics and statistical mechanics, in particular) examines properties and enumerations of domino tilings. See, for example, \cite{cohn kenyon propp, elkies kuperberg larsen propp, kasteleyn, stanley domino}, among many others. The range of questions, regions, machinery, restrictions, and so on, is quite broad. Here, we focus our attention on a small selection of problems, with the goal of drawing attention to what could be an interesting area of research.

\begin{definition}
A \emph{domino} is a $1 \times 2$ rectangle. A \emph{domino tiling} of a region $R$ is a tiling of $R$ by dominoes.
\end{definition}

Of course, not every region can be tiled by dominoes. For example, such a region must have sides of integer lengths, only right angles (internal or external) along its boundary, and its area must be even. Throughout this section, assume that all regions are tileable by dominoes.

In light of the tiling-based isoperimetric analysis of Sections~\ref{sec:long element} and~\ref{sec:general perm}, it is interesting to consider such a perspective for domino tilings. In this section, we present a smattering of results about perimeter dominoes, with the hope that their range will inspire further research into the many remaining questions on this topic. Proofs are kept brief, when they do appear.

Recall Definition~\ref{defn:perimeter tile generally}. which designated a tile as a (strong) perimeter tile based on the number of edges it shared with the boundary of the region, and the paragraph preceding it.

\begin{definition}\label{defn:perimeter domino}
Consider a domino tiling of a region $R$ and a domino $d$ in that tiling. If the domino $d$ shares a length-$2$ edge with the boundary of $R$, then $d$ is a \emph{(strong) perimeter domino}. Again, we distinguish these from \emph{weak perimeter dominoes}, which share a positive-length section of boundary with the boundary of the region but which are not strong.
\end{definition}

Definition~\ref{defn:perimeter domino} captures when a choice of domino orientation leads to greater coincidence with the boundary than might otherwise have occurred. For example, a domino positioned in a ``convex'' corner of a region would necessarily share a contiguous length-$2$ section of its boundary with the boundary of $R$, as shown in Figure~\ref{fig:corner domino}, but that fact does not depend on the orientation of that domino.

\begin{figure}[htbp]
\begin{tikzpicture}[scale=.5]
\draw[densely dashed] (1,1) -- (2,1) -- (2,0) -- (0,0);
\draw (0,0) -- (0,1) -- (1,1);
\foreach \x in {(0,0), (0,1), (1,1)} {\fill \x circle (3pt);}
\draw[decorate,decoration={snake,amplitude=1mm,segment length=4mm,pre length=1mm,post length=.5mm}] (1,1) to[bend left] (3,2) to[bend left] (5,-3) to[bend left] (.5,-6) to[bend left] (-4,-3) to[bend left] (0,0);
\end{tikzpicture}
\caption{A domino in a convex corner of a region necessarily shares a contiguous length-$2$ section of its boundary with the region.}
\label{fig:corner domino}
\end{figure}
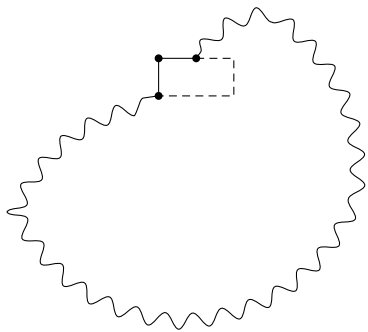

\begin{definition}
Let $R_{m,n}$ denote the $m\times n$ rectangle. Because we will use it frequently, set $R_n := R_{2,n}$. For a region $R$, let $\mathcal{D}(R)$ be the set of domino tilings of $R$.
\end{definition}

Not only is the region $R_n$ of manageable size and complexity, it also has attractive enumerative properties and lends itself well to recursive procedures. For example, it is an easy exercise to show that domino tilings of $R_n$ are enumerated by the Fibonacci number $f(n+1)$, where $f(0) = 0$ \cite[A000045]{oeis}.

We begin our analysis of perimeter tiles in $R_n$ with an analogue to Theorems~\ref{thm:lower bound for total perimeter tiles} and~\ref{thm:upper bound for total perimeter tiles}. As in previous sections, we restrict to $n > 1$ to avoid overcounting the one domino in the one tiling of $R_1$. We recycle the function $\perim$ from Definition~\ref{defn:perimeter counting functions} to count perimeter dominoes in a tiling $D \in \mathcal{D}(R_n)$.

\begin{proposition}\label{prop:2xn domino bounds}
For $n \ge 2$,
$$\min\{\perim(D) : D \in \mathcal{D}(R_n)\} = 2 \hspace{.25in} \text{and} \hspace{.25in} \max\{\perim(D) : D \in \mathcal{D}(R_n)\} = n.$$
\end{proposition}

\begin{proof}
In any of the four corners of the region $R_n$, either vertical or horizontal placement of the domino occupying that corner will produce a perimeter tile. The two corners along an edge of length two may be covered by the same (perimeter) domino, and there is no way to cover more than two corners by a single domino. Therefore there must be at least two perimeter dominoes in any tiling. Orienting all dominoes to be parallel to the length-$2$ edge of $R_n$ obtains this minimal value.

If, in contrast, all dominoes are oriented perpendicular to the length-$2$ edge of $R_n$, with the first domino having the opposite orientation if $n$ is odd, then all $n$ dominoes will be perimeter tiles.
\end{proof}

Examples of the tilings described in the proof of Proposition~\ref{prop:2xn domino bounds} are shown in Figure~\ref{fig:domino bounds}.

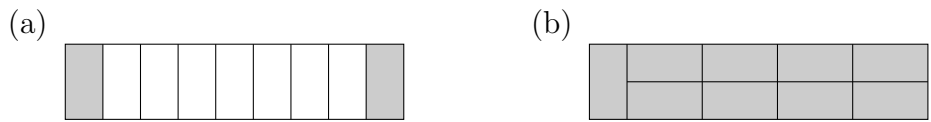
\begin{figure}[htbp]
\begin{tikzpicture}[scale=.5]
\fill[black!20] (0,0) rectangle (1,2);
\fill[black!20] (8,0) rectangle (9,2);
\draw (0,0) rectangle (9,2);
\foreach \x in {1,...,8} {\draw (\x,0) -- (\x,2);}
\draw (-1,2.5) node {(a)};
\end{tikzpicture}
\hspace{.5in}
\begin{tikzpicture}[scale=.5]
\fill[black!20] (0,0) rectangle (9,2);
\draw (0,0) rectangle (9,2);
\draw (1,0) -- (1,2);
\draw (1,1) -- (9,1);
\foreach \x in {3,5,7} {\draw (\x,0) -- (\x,2);}
\draw (-1,2.5) node {(b)};
\end{tikzpicture}
\caption{Domino tilings of $R_9$ having (a) the least and (b) the greatest number of perimeter dominoes. Perimeter dominoes are shaded.}
\label{fig:domino bounds}
\end{figure}

In analogy to the work of Section~\ref{sec:long element}, we now study the total number of perimeter dominoes appearing among all domino tilings of $R_n$.

\begin{definition}
For $n > 2$, let $P_n$ be the number of perimeter dominoes appearing among all elements of $\mathcal{D}(R_n)$.
\end{definition}

\begin{proposition}\label{prop:domino total}
For $n > 2$,
$$P_{n+1} = P_n + P_{n-1} + 2f(n),$$
where $f(n)$ is the $n$th Fibonacci number, and $P_2 = 4$ and $P_3 = 8$. In closed form,
$$P_n = \frac{4nf(n-1) + (2n+8)f(n)}{5},$$
and the generating function is
$$\sum\limits_{n \ge 2} P_nx^n = \frac{4x^2 - 4x^4 - 2x^5}{(1-x-x^2)^2}.$$
\end{proposition}

\begin{proof}
To compute $P_2$ and $P_3$, we look at the two elements of $\mathcal{D}(P_2)$ and the three elements of $\mathcal{D}(P_3)$, shown in Figure~\ref{fig:dominoes in small rectangles}. 
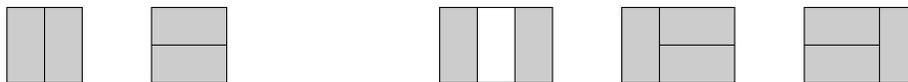
\begin{figure}[htbp]
\begin{tikzpicture}[scale=.5]
\fill[black!20] (0,0) rectangle (2,2);
\draw (0,0) rectangle (2,2);
\draw (1,0) -- (1,2);
\end{tikzpicture}
\hspace{.25in}
\begin{tikzpicture}[scale=.5]
\fill[black!20] (0,0) rectangle (2,2);
\draw (0,0) rectangle (2,2);
\draw (0,1) -- (2,1);
\end{tikzpicture}
\hspace{1in}
\begin{tikzpicture}[scale=.5]
\fill[black!20] (0,0) rectangle (1,2);
\fill[black!20] (2,0) rectangle (3,2);
\draw (0,0) rectangle (3,2);
\draw (1,0) -- (1,2);
\draw (2,0) -- (2,2);
\end{tikzpicture}
\hspace{.25in}
\begin{tikzpicture}[scale=.5]
\fill[black!20] (0,0) rectangle (3,2);
\draw (0,0) rectangle (3,2);
\draw (1,0) -- (1,2);
\draw (1,1) -- (3,1);
\end{tikzpicture}
\hspace{.25in}
\begin{tikzpicture}[scale=.5]
\fill[black!20] (0,0) rectangle (3,2);
\draw (0,0) rectangle (3,2);
\draw (2,0) -- (2,2);
\draw (0,1) -- (2,1);
\end{tikzpicture}
\caption{All domino tilings of $R_2$ and $R_3$, with shaded perimeter dominoes.}
\label{fig:dominoes in small rectangles}
\end{figure}
For clarity, we will draw all rectangles $R_n$ so that the side of length $n$ is horizontal. It will help us to consider the sequence $\{a_n\}$, which we define to count perimeter dominoes that are not a leftmost vertical domino, appearing among all elements of $\mathcal{D}(R_n)$. Thus, for example, we see in the figures above that $a_2 = 3$ and $a_3 = 6$.

There are two possible ways to tile the leftmost column of $R_{n+1}$: by a vertical domino or by two horizontal dominoes. In each case, the placed dominoes are perimeter tiles. From these options, we see that
\begin{align*}
P_{n+1} &= \Big(1\cdot|\mathcal{D}(R_n)| + a_n\Big) + \Big(2\cdot|\mathcal{D}(R_{n-1})| + a_{n-1}\Big)\\
&= 1\cdot f(n+1) + a_n + 2\cdot f(n) + a_{n-1}.
\end{align*}
By the same breakdown of cases, we have
$$a_{n+1} = 0 \cdot f(n+1) + a_n + 2\cdot f(n) + a_{n-1}.$$
Therefore $P_{n+1} = a_{n+1} + f(n+1)$, and so
\begin{align*}
P_n + P_{n-1} &= a_n + f(n) + a_{n-1} + f(n-1)\\
&= a_{n+1} - 2f(n) + f(n+1)\\
&= P_{n+1} - 2f(n).
\end{align*}
The closed form can be extracted by iterating the recursion to find a Fibonacci convolution, and the associated generating function can be found by standard methods, such as those outlined in \cite{wilf}.
\end{proof}

The values $P_n$, along with $|\mathcal{D}(R_n)|$, are presented in Table~\ref{table:domino data} for $n \in [2,9]$. The sequence $\{P_n\}$ is entry A320947 of \cite{oeis}.

\begin{table}[htbp]
$$\begin{array}{r||c|c}
\raisebox{0in}[.2in][.1in]{}n & \ \hspace{.25in}P_n\hspace{.25in} \ & |\mathcal{D}(R_n)| = f(n+1)\\
\hline
\hline
\raisebox{0in}[.2in][.1in]{} 2 & 4 & 2\\
\hline
\raisebox{0in}[.2in][.1in]{} 3 & 8 & 3\\
\hline
\raisebox{0in}[.2in][.1in]{} 4 & 16 & 5\\
\hline
\raisebox{0in}[.2in][.1in]{} 5 & 30 & 8\\
\hline
\raisebox{0in}[.2in][.1in]{} 6 & 56 & 13\\
\hline
\raisebox{0in}[.2in][.1in]{} 7 & 102 & 21\\
\hline
\raisebox{0in}[.2in][.1in]{} 8 & 184 & 34\\
\hline
\raisebox{0in}[.2in][.1in]{} 9 & 328 & 55\\
\end{array}$$
\caption{The total number of perimeter dominoes, along with the number of domino tilings, of a $2\times n$ rectangle, for $n \in [2,9]$.}
\label{table:domino data}
\end{table}

One can, similarly, refine some of this data by calculating the number of domino tilings of $R_n$ that have $k$ perimeter dominoes:
\begin{align*}
\#\{D \in \mathcal{D}(R_n) : \perim(D) = k\} &= \binom{n-\lfloor \frac{k+1}{2}\rfloor -1}{\lceil \frac{k-3}{2} \rceil} + \binom{n-\lceil \frac{k+1}{2}\rceil -1}{\lfloor \frac{k-3}{2} \rfloor}\\
&=
\begin{cases}
\displaystyle{2\binom{n-m-2}{m-1}} & \text{ if } k = 2m+1, \text{ and}\\
\displaystyle{\binom{n-m-1}{m-1} + \binom{n-m-2}{m-2}} & \text{ if } k = 2m.
\end{cases}
\end{align*}
This data is sequence A073044 of \cite{oeis}.

Given the motivation of this work, it is natural to look at isoperimetric properties in the setting of $\mathcal{D}(R_{m,n})$. 

\begin{lemma}\label{lem:plump rectangles in between}
If $m,n > 2$, then
$$\min\{\perim(D) : D \in \mathcal{D}(R_{m,n})\} \ge 4$$
and
$$\max\{\perim(D) : D \in 
\mathcal{D}(R_{m,n})\} \le m+n-2 < \frac{mn}{2}.$$
\end{lemma}

\begin{proof}
Because $m$ and $n$ are both greater than $2$, no single domino can cover two corners of the rectangle $R_{m,n}$. Therefore $\min\{\perim(D) : D \in \mathcal{D}(R_{m,n})\} \ge 4$.

As in the proof of Proposition~\ref{prop:at most n perimeter rhombi}, we note that at most $(2m+2n-4)/2 = m+n-2$ perimeter dominoes can be used to cover the $2m+2n-4$ squares along the boundary of $R_{m,n}$. Because $m,n>2$, this value is strictly less than $mn/2$.
\end{proof}

Note the implications of Lemma~\ref{lem:plump rectangles in between} for isoperimetric-type questions about rectangles.

\begin{corollary}
Among all rectangles of a fixed area, the dimensions that can achieve both the least and the greatest number of perimeter dominoes is the $2\times n$ rectangle.
\end{corollary}

We close this section by looking at perimeter dominoes in non-rectangular regions. We first observe that different types of corners in a region can force different kinds of perimeter behavior.

\begin{definition}
A \emph{convex} corner is one in which the interior angle is $90^{\circ}$. A \emph{non-convex} corner is one in which the interior angle is $270^{\circ}$. 
\end{definition}

Convex and non-convex corners are illustrated in Figure~\ref{fig:corners}, with filled circles indicating convex corners and open circles indicating non-convex corners. 
\begin{figure}[htbp]
\begin{tikzpicture}[scale=.5]
\draw[fill=black!20] (0,0) -- (2,0) -- (2,2) -- (8,2) -- (8,4) -- (4,4) -- (4,6) -- (6,6) -- (6,10) -- (0,10) -- (0,8) -- (-2,8) -- (-2,4) -- (0,4) -- (0,0);
\draw[fill=white] (0,5) -- (3,5) -- (3,7) -- (1,7) -- (1,6) -- (0,6) -- (0,5);
\foreach \x in {(0,0), (2,0), (8,2), (8,4), (6,6), (6,10), (0,10), (-2,8), (-2,4)} {\fill[black] \x circle (4pt);}
\foreach \x in {(1,6)} {\fill[black] \x circle (4pt);}
\foreach \x in {(0,5), (3,5), (3,7), (1,7), (0,6), (0,5)} {\filldraw[fill=white] \x circle (4pt);}
\foreach \x in {(2,2), (4,4), (4,6), (0,8), (0,4)} {\filldraw[fill=white] \x circle (4pt);}
\path[->] (8,7) edge[bend right] (5,8);
\draw (8,7) node[right] {region};
\draw[white] (-2,7) node[left] {region};
\end{tikzpicture}
\caption{A rectilinear region with convex corners marked by filled circles and non-convex corners marked by open circles.}
\label{fig:corners}
\end{figure}
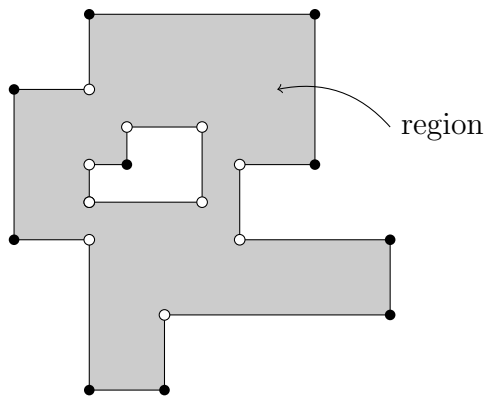
While the lines in Figure~\ref{fig:corners} represent the (disconnected) boundary of the region, tile behavior at the two types of corners has different impacts on the collection of perimeter tiles.

\begin{lemma}\label{lem:non-convex doesn't make perimeter}
Suppose edges of length at least two meet at the corner of a region. If that corner is convex, then it must be incident to a perimeter tile. Non-convex corners, in general, need not be incident to perimeter tiles.
\end{lemma}

This lemma is illustrated in Figure~\ref{fig:tiles at corners}, where, again, filled circles indicate convex corners and open circles indicate non-convex corners.

\begin{figure}[htbp]
\begin{tikzpicture}[scale=.5]
\draw (-.5,0) -- (2,0) -- (2,2.5);
\draw[dashed] (0,0) -- (0,1) -- (2,1);
\draw (3.5,0) -- (6,0) -- (6,2.5);
\draw[dashed] (5,0) -- (5,2) -- (6,2);
\foreach \x in {(2,0),(6,0)} {\fill[black] \x circle (4pt);}
\draw (10,-1.5) -- (10,1) -- (12.5,1);
\draw[dashed] (11,1) -- (11,3) -- (10,3) -- (10,1) -- (8,1) -- (8,0) -- (10,0);
\filldraw[fill=white] (10,1) circle (4pt);
\end{tikzpicture}
\caption{How tiles can behave around convex and non-convex corners.}
\label{fig:tiles at corners}
\end{figure}
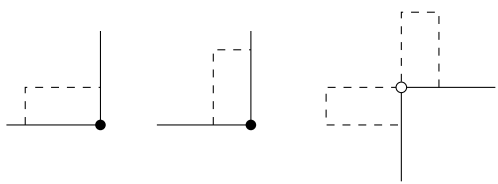

Lemma~\ref{lem:non-convex doesn't make perimeter} suggests that a holey square (see \cite{tenner holey}) need not have more perimeter tiles than a non-holey square. Compare Lemma~\ref{lem:plump rectangles in between} to Figure~\ref{fig:holey and non-holey}, where both regions have minimally many (four) perimeter tiles.

\begin{figure}[htbp]
\begin{tikzpicture}[scale=.5]
\foreach \x in {(0,0), (0,5), (4,0), (4,5)} {\fill[black!20] \x rectangle ++(2,1);}
\draw[thick] (0,0) rectangle (6,6);
\foreach \x in {2,3,4} {\draw (\x,0) -- (\x,6);}
\foreach \x in {1,2,3,4,5} {\draw (0,\x) -- (2,\x); \draw (4,\x) -- (6,\x);}
\foreach \x in {2,4} {\draw (2,\x) -- (4,\x);}
\end{tikzpicture}
\hspace{.5in}
\begin{tikzpicture}[scale=.5]
\foreach \x in {(0,0), (0,5), (4,0), (4,5)} {\fill[black!20] \x rectangle ++(2,1);}
\draw[thick] (0,0) rectangle (6,6);
\draw[thick] (2,2) rectangle (4,4);
\foreach \x in {2,4} {\draw (\x,0) -- (\x,6);}
\draw (3,0) -- (3,2);
\draw (3,4) -- (3,6);
\foreach \x in {1,2,3,4,5} {\draw (0,\x) -- (2,\x); \draw (4,\x) -- (6,\x);}
\foreach \x in {2,4} {\draw (2,\x) -- (4,\x);}
\end{tikzpicture}
\caption{A domino tiling of the $6\times 6$ square, and a domino tiling of the $6\times 6$ holey square missing a $2\times 2$ central region, both of which have four (shaded) perimeter tiles.}\label{fig:holey and non-holey}
\end{figure}
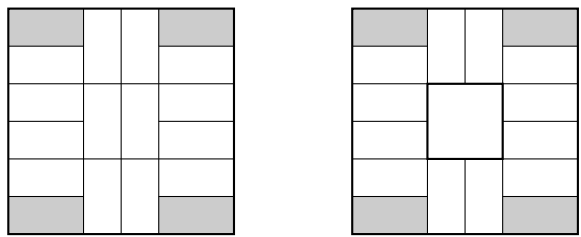

As we saw in Lemma~\ref{lem:plump rectangles in between}, the minimal tile perimeter of a rectangle is four. This poses an obvious question: are there nontrivial (necessarily non-rectangular) regions with tile perimeter less than four? Indeed, there are! Then, is there any region that can be tiled with no perimeter tiles? In fact, there is not!

We show this by illustrating, much like in Lemma~\ref{lem:non-convex doesn't make perimeter}, that certain formations will force the existence of a perimeter tile. We use the notation X$^*$ to indicate that the string X appears zero or more times.

\begin{definition}
Let $\sf{N}$, $\sf{S}$, $\sf{E}$, and $\sf{W}$ refer to the cardinal directions. Consider a region with a segment of boundary that can be described, up to rotation, by the string
$$\sf{WN}(\sf{N}^*\sf{E}^*)^*\sf{ES},$$
oriented so that the first two letters form a convex corner of the region. That segment of the boundary is a \emph{cloverleaf} of the region.
\end{definition}

A cloverleaf described by $\sf{WNNNENNEENEEEES}$ is illustrated in Figure~\ref{fig:cloverleafs}.

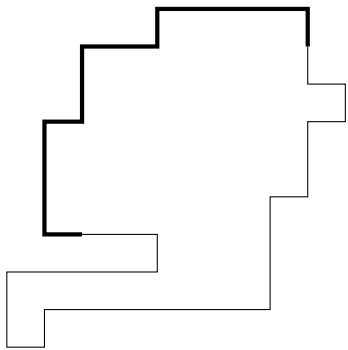
\begin{figure}[htbp]
\begin{tikzpicture}[scale=.5]
\draw[ultra thick] (0,0) -- ++(-1,0) --++(0,3) --++(1,0) --++(0,2) --++(2,0) --++(0,1) --++(4,0) --++(0,-1);
\draw (0,0) --++(2,0) --++(0,-1) --++(-4,0) --++(0,-2) --++(1,0) --++(0,1) --++(6,0) --++(0,3) --++(1,0) --++(0,2) --++(1,0) --++(0,1) --++(-1,0) --++(0,1);
\end{tikzpicture}
\caption{A region with four cloverleafs, one of which is marked.}\label{fig:cloverleafs} 
\end{figure}

The following claims are not hard to prove, and we leave them to the reader.

\begin{lemma}\label{lem:cloverleafs}\
\begin{enumerate}\renewcommand{\labelenumi}{(\alph{enumi})}
\item Every finite region has at least one cloverleaf of each of the four orientations.
\item In any tiling, every cloverleaf has at least one perimeter tile.
\end{enumerate}
\end{lemma}

Although Lemma~\ref{lem:cloverleafs}(a) says that every region has at least four types of cloverleafs, some might overlap. However, in any sufficiently large region, two types that differ by a $180^{\circ}$ rotation cannot overlap enough to share any perimeter tiles. This, combined with Lemma~\ref{lem:cloverleafs}(b) produces the desired result. In fact, one can show, inductively, that there must be two \emph{parallel} perimeter tiles.

\begin{corollary}\label{cor:at least 2 perimeter tiles}
Suppose that $R$ is a region of area greater than $2$, which can be tiled by dominos. In any domino tiling of $R$, there are at least two perimeter tiles. Moreover, up to orientation of $R$, there is a tiling with a pair of vertical perimeter dominoes with one sharing its leftmost edge with the boundary of $R$ and the other of sharing its rightmost  edge with the boundary of $R$.
\end{corollary}

We bring this section to a close by describing all simply connected (that is, non-holey) regions with exactly two perimeter tiles. This relies heavily on Lemma~\ref{lem:non-convex doesn't make perimeter} and Corollary~\ref{cor:at least 2 perimeter tiles}.

\begin{corollary}\label{cor:minimal perimeter}
Suppose that $R$ is a simply connected region that can be tiled by dominos, and that $R$ has a tiling with exactly two perimeter tiles. Then, up to rotation, $R$ must be as depicted in Figure~\ref{fig:2 perimeter tiles}, and every tile incident to the perimeter of $R$ in a minimal-perimeter tiling is a vertical domino.
\end{corollary}

\begin{figure}[htbp]
\begin{tikzpicture}[scale=.5]
\draw[thick, blue] (0,0) -- (0,2);
\draw[thick, blue] (13.5,2) -- (13.5,4);
\foreach \x in {(1.5,-1),(1.5,2),(3,-2),(3,3),(4.5,4),(7.5,-2),(9,-1),(10.5,0),(12,1),(12,4)} {\draw[thick, red] \x --++(0,1);}
\foreach \x in {(0,0),(0,2),(1.5,3),(1.5,-1),(3,4),(7.5,-1),(9,0),(10.5,1),(12,2),(12,4)} {\draw[thick] \x --++(1.5,0);}
\draw[thick] (4.5,5) -- (12,5);
\draw[thick] (3,-2) -- (7.5,-2);
\end{tikzpicture}
\caption{A region that can be tiled to have exactly two perimeter tiles. Blue segments must have length $2$, red segments must have length $1$, and the number of corners defined by red and black edges is unrestricted. Black segments have arbitrary positive lengths.}\label{fig:2 perimeter tiles}
\end{figure}
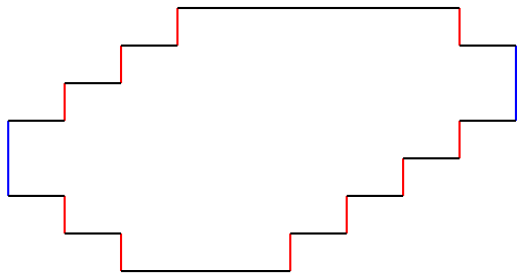

A specific example of Corollary~\ref{cor:minimal perimeter} appears in Figure~\ref{fig:2 perimeter example}.

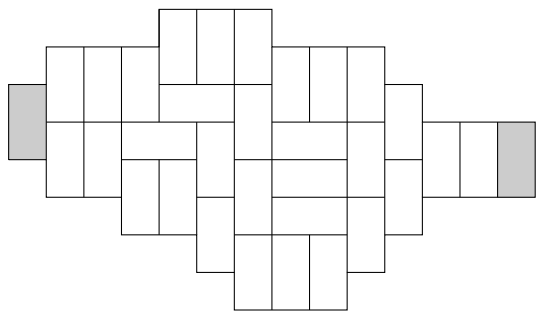
\begin{figure}[htbp]
\begin{tikzpicture}[scale=.5]
\fill[black!20] (0,0) rectangle (1,2);
\fill[black!20] (13,-1) rectangle (14,1);
\draw (0,0) -- (0,2) -- (1,2) -- (1,3) -- (4,3) -- (4,4) -- (7,4) -- (7,3) -- (10,3) -- (10,2) -- (11,2) -- (11,1) -- (14,1) -- (14,-1) -- (11,-1) -- (11,-2) -- (10,-2) -- (10,-3) -- (9,-3) -- (9,-4) -- (6,-4) -- (6,-3) -- (5,-3) -- (5,-2) -- (3,-2) -- (3,-1) -- (1,-1) -- (1,0) -- (0,0);
\draw (1,0) -- (1,2);
\draw (2,-1) -- (2,3);
\draw (3,-1) -- (3,3);
\draw (4,-2) -- (4,0); \draw (4,1) -- (4,4);
\draw (5,-2) -- (5,1); \draw (5,2) -- (5,4);
\draw (6,-3) -- (6,4);
\draw (7,-4) -- (7,3);
\draw (8,-4) -- (8,-2); \draw (8,1) -- (8,3);
\draw (9,-3) -- (9,3);
\draw (10,-2) -- (10,2);
\draw (11,-1) -- (11,1);
\draw (1,1) -- (6,1);
\draw (3,0) -- (5,0);
\draw (4,2) -- (7,2);
\draw (5,-1) -- (6,-1);
\draw (6,-2) -- (9,-2);
\draw (6,0) -- (9,0);
\draw (7,-1) -- (10,-1);
\draw (7,1) -- (10,1);
\draw (10,0) -- (11,0);
\draw (12,-1) -- (12,1);
\draw (13,-1) -- (13,1);
\end{tikzpicture}
\caption{A tiling of a region such that there are exactly two (shaded) perimeter tiles.}\label{fig:2 perimeter example}
\end{figure}

\section{Directions for further research}\label{sec:further research}

Regarding the rhombic tilings of Elnitsky polygons, it would be interesting to have analogues of the work in Section~\ref{sec:long element} for arbitrary permutations. The significance of such results to the combinatorics of reduced decompositions inspires questions in other directions, too. What do these properties tell us about structure in Coxeter graphs, or in the contracted graphs of \cite{bergeron-ceballos-labbe}? What do they imply for enumerative and structural features of the Bruhat order? Elnitsky's work extends to Coxeter groups of types $B$ and $D$, as well. What can we say about perimeter tiles in those settings?

From Proposition~\ref{prop:domino total}, we get a sense of the average number of perimeter dominoes in a randomly chosen domino tiling of a $2\times n$ rectangle. One could look at analogues of Proposition~\ref{prop:domino total} for arbitrary rectangles. One could similarly pursue the circumstances of perimeter dominoes that appear in domino tilings of non-rectangular regions, as suggested in the latter portion of Section~\ref{sec:domino}, including regions with nontrivial topology, like the ``holey square'' of \cite{tenner holey}.

In this work, we have only looked at two types of tilings, but, of course, there are infinitely many tiling-style problems, many of which have compelling mathematics. There are also scenarios that might be rephrased as a tiling-based isoperimetric question. What do our methods have to say in those applications? Do strong perimeter tiles bear extra significance there? Are weak perimeter tiles mathematically noteworthy, too? Can these tiles (either type) be enumerated, either directly or probabilistically?

\section*{Acknowledgements}

I am grateful for the detailed and thoughtful advice of an anonymous referee.

\end{document}